\documentclass[10pt,reqno]{amsart}
\usepackage[utf8]{inputenc}
\usepackage{amsthm}
\usepackage{tikz}
\usetikzlibrary{arrows.meta}
\usetikzlibrary{shapes.geometric}
\usepackage{forest}
\usepackage{wrapfig}
\usepackage{caption}
\usepackage{float}
\usepackage{units}
\usepackage{todonotes}
\usepackage{comment}
\usepackage{hyperref}
\usepackage[initials,msc-links]{amsrefs}

\newcounter{indice}

\newcommand{\cyclefig}[1]{\begin{tikzpicture}[scale= 0.35]
        \cycle{#1}
      \end{tikzpicture} \hspace{2.7mm}\vspace{0.9mm}}

\newcommand{\cycle}[1]{ 
\setcounter{indice}{0};
\foreach \i in {#1}
\addtocounter{indice}{1};
\addtocounter{indice}{1}
\draw [help lines] (0,0) grid (\theindice-1,\theindice-1);
\setcounter{indice}{1};
\foreach \i in { #1 } { 
\draw (\theindice-.5,\i-.5) [fill] circle (.18);
\draw[dashed] (\theindice-.5, \theindice-.5)--(\theindice-.5, \i-.5);
\draw[dashed] (\i-.5, \i-.5)--(\theindice-.5, \i-.5);
\addtocounter{indice}{1};
}
\addtocounter{indice}{-1};
}

\usetikzlibrary{shapes.geometric}

\newtheorem{theorem}{Theorem}[section]
\newtheorem{example}[theorem]{Example}
\newtheorem{remark}[theorem]{Remark}
\newtheorem{lemma}[theorem]{Lemma}
\newtheorem{prop}[theorem]{Proposition}
\newtheorem{definition}[theorem]{Definition}

\title{Unknotted Cycles}
\author{Christopher Cornwell and Nathan McNew }

\begin{document}

\begin{abstract}
    Noting that cycle diagrams of permutations visually resemble grid diagrams used to depict knots and links in topology, we consider the knot (or link) obtained from the cycle diagram of a permutation.  We show that the permutations which correspond in this way to an unknot are enumerated by the Schr\"{o}der numbers, and also enumerate the permutations corresponding to an unlink.  The proof uses Bennequin's inequality.
\end{abstract}

\maketitle

\section{Introduction}\label{sec:intro}

A convenient way to visualize a permutation is to draw a plot of the permutation on an $n\times n$ grid, placing a dot in the box at each of the locations $(i,\sigma(i))$.  The permutation's cycle structure can be represented by making the plot to be a \textit{cycle diagram} \cite{Elizalde}. At each index $i$ we draw a vertical line from $(i,i)$ to the point $(i,\sigma(i))$, followed by a horizontal line to $(\sigma(i),\sigma(i))$.  If $i$ is a fixed point, $i=\sigma(i)$, no additional lines are drawn.  The result is a diagram in which the cycles of the permutation can be traced out along the lines of the diagram in a natural way.  

\begin{wrapfigure}{r}{0.37\textwidth}
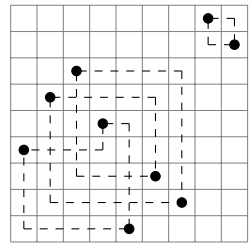

\centering
\hspace{1mm}\cyclefig{4,6,7,5,1,3,2,9,8}
\caption{The cycle diagram \\ of $\pi=467513298$.}
    \label{fig:467513298}
\end{wrapfigure}

For example, from the cycle diagram for the permutation $\pi=467513298$ (written in one line notation) depicted in Figure \ref{fig:467513298}, one can readily identify the cycle decomposition $\pi=(145)(2637)(89)$ (written in cycle notation) by tracing out the lines of the diagram.  
%Henceforth we will always write our permutations using one line notation.  

Note that the only corners in a cycle diagram occur at the plotted points of the permutation and along the line $y{=}x$.  The appearance of these diagrams strongly resembles grid diagrams which are a useful tool in the study of knots in topology.

Formally, a grid diagram is an $n \times n$ lattice where
each row, and each column, has exactly two marked boxes (traditionally marked $X$ and $O$) and a line is drawn between the marked boxes in each row and column.  The diagram is interpreted as a knot (an embedding of $S^1$ in $\mathbb{R}^3$) or a link (multiple copies of $S^1$) by designating all of the vertical lines to be overcrossing and the horizontal lines as undercrossing.  (Section \ref{sec:knots_intro} provides an overview of necessary basic notions from knot theory, see also  \cites{Crom, Dyn, NT}.)

Crucially, the only distinctions between diagrams that are valid cycle diagrams and those that are valid grid diagrams are:
\begin{itemize}
    \item In cycle diagrams one of the two designated points in each row/column must lie on the $y=x$ line, which isn't necessarily the case for grid diagrams.
    \item In grid diagrams, it is not allowed to have a single point in a row or column as occur in cycle diagrams when there are fixed points.
\end{itemize}
In light of the second point, so long as we take a permutation without fixed points (called a \textit{derangement}) then by drawing its cycle diagram and interpreting it as the grid diagram of a link we can build a link corresponding to any derangement.  We refer to this link as the \textit{link associated to a permutation} or, when the permutation is a cycle, as the \textit{knot associated to a cycle}.  We do not associate a link to a permutation that is not a derangement.  

Many natural questions arise, including:
\begin{enumerate}
    \item Which knots (links) are associated to some derangement?
    \item How many different derangements are associated with a given link?
\end{enumerate}  
While we do not have a complete answer to question (1), the answer is certainly not ``all knots'' (or links), as explained at the end of Section \ref{subsec:link-diagrams}.

As to question (2), in this paper we enumerate the cycles that are associated to the unknot (called an \textit{unknotted cycle}) as well as permutations associated to an unlink (\textit{unlinked permutations}).  Our main result is the following.

\begin{theorem} \label{thm:main}
The unknotted cycles are enumerated by the shifted sequence of (large) Schr\"{o}der numbers $S_n$.  Enumerating these numbers as $S_1=1$, $S_2=2$, \ldots The number of unknotted cycles of length $n$ is $S_{n-1}$. 
\end{theorem}

The Schr\"{o}der numbers count a wide array of combinatorial objects, notably seperable permutations of length $n$, and lattice paths from (1,1) to $(n,n)$ consisting of north (0,1), east (1,0) and northeast (1,1) steps which never go above the diagonal \cite{OEISschroder}.  Their generating function is $S(x) = \sum_{n=1}^\infty S_n x^n = \frac{1}{2}\left(1-x-\sqrt{1-6x+x^2}\right)$, which satisfies the recursion $S(x)=x+xS(x)+S(x)^2$.  Asymptotically \cite[Ex. 2.2.1-12]{knuth} \begin{equation}
    S_n \sim \frac{\sqrt{2}-1}{2^{3/4}\sqrt{\pi}}(3+\sqrt{8})^nn^{-3/2}. \label{eq:knuth}
\end{equation}

The relationship between cycle diagrams and grid diagrams does not appear to have been considered before in the literature.  A related body of work is the study of random knots, particularly via the \emph{random grid model} \cite{Zohar}.  A common question in this area considers the probability of a random knot being equivalent to a fixed knot $K$. Particularly, for each given integer $n>0$, random knot models select a knot $K_n$; as $n\to\infty$, does the probability that $K_n$ is equivalent to $K$ approach zero, and what are the asymptotics? A broadly construed conjecture is the Frisch-Wasserman-Delbr\"uck conjecture (see \cite{Zohar} for discussion).

In the \emph{random grid model}, $K_n$ is given by selecting a random pair of $n$-cycles $(\sigma,\tau)$, independently and uniformly. Let $\sigma$ be $(\sigma_1,\ldots,\sigma_n)$ in cycle notation (so $\sigma(\sigma_i)=\sigma_{i+1}$ for each $i$, $\sigma_{n+1}=\sigma_1=1$, say), and likewise for $\tau$.  A grid diagram is then determined by drawing, for each $1\le i\le n$, a vertical line from $(\sigma_i,\tau_i)$ to $(\sigma_i,\tau_{i+1})$, and from there a horizontal line to $(\sigma_{i+1},\tau_{i+1})$. Then $K_n$ is the knot of this grid diagram (a {knot} since $\sigma$ and $\tau$ a $n$-cycles). In this model, Witte has shown that the probability of $K_n$ being a given knot (e.g.\ the unknot), is $\mathcal O(n^{-1/10})$ as $n\to\infty$ (this is implied by Theorem 6.0.1 of \cite{Witte}).

When $\sigma=\tau$, the grid diagram in the above model is the cycle diagram of $\sigma$. So, on the diagonal of the random grid model, Theorem \ref{thm:main} gives an exact probability of $K_n$ being the unknot, equal to $\frac{S_{n-1}}{(n-1)!} \sim \frac{5\sqrt{2}-7}{2^{5/4} \pi n }\left(\frac{e(3+\sqrt{8})}{n}\right)^{n}$ using \eqref{eq:knuth} and Stirling's formula.  Hence, the probability of $K_n$ being the unknot decays super-exponentially as $n\to\infty$.

We show that unknotted cycles are counted by the Schr\"oder numbers by establishing a bijection between them and the rooted-signed-binary trees defined in Section \ref{sec:trees}.  We first define a way to construct an unknotted cycle of size $n+1$ from a rooted-signed-binary tree of size $n$ in Section \ref{sec:bijection}, then show in Section \ref{sec:welldef} that it is well defined and one-to-one on equivalency classes of these trees. Finally, in Section \ref{sec:topology} we use results from topology to prove that it is surjective {--} all unknotted cycles can be obtained in this way.  Finally, we enumerate the unlinked permutations in Section \ref{sec:links} and note a potential relationship to the Diaconis-Graham inequality.

%%% Characterize those knots which could possibly be the knot associated to a cycle?

\section{Knots and Links} \label{sec:knots_intro}
Informally, in this article, a \textit{knot} is a closed curve in $\mathbb R^3$ which has no self-intersections.  In addition, two knots $K$ and $K'$ are considered to be \textit{equivalent} if there is a continuous deformation that takes $K$ to $K'$, such that, at each time during the deformation of $K$, it is a knot.  More precisely, a knot $K$ is the image of a piecewise $\mathcal C^1$-embedding, $S^1\hookrightarrow \mathbb R^3$, and $K$ is equivalent to a knot $K'$ if there is an ambient isotopy (see \cite{Lickorish}) carrying one to the other.  When there are multiple knots, no two of which intersect, the multi-curve is called a \textit{link}.  The notion of equivalent links is analogous to equivalent knots.  We use ``links'' inclusively, so a link could be a knot (having just one component curve).

A link is \textit{oriented} by choosing a consistent positive tangent along each component curve; in figures the positive tangent will be indicated by an arrow. %\todo{do we need orientation? yes, it's needed to give a sign to crossings, which we use for a couple of things}

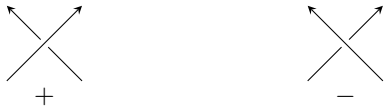
\begin{figure}[ht]
    \begin{tikzpicture}[>=stealth]
        \draw[->]
            (1,0)--(0.55,0.45)
            (0.45,0.55)--(0,1);
        \draw[->]
            (0.5,-0.2) node {$+$}
            (0,0)--(1,1);
        \draw[->]
            (4,0)--(4.45,0.45)
            (4.55,0.55)--(5,1);
        \draw[->]
            (4.5,-0.2) node {$-$}
            (5,0)--(4,1);
%       \draw[->]
%           (6,0) .. controls (6.5,0.5) .. (6,1);
%       \draw[->]
%           (6.5,-0.2) node {$K_0$}
%           (7,0) .. controls (6.5,0.5) .. (7,1);
    \end{tikzpicture}
    \caption{Positive and negative crossings}
    \label{fig:crossing_sign}
\end{figure}

\subsection{Link diagrams}\label{subsec:link-diagrams} Often, knots and links are studied via ``sufficiently generic'' projections to a plane, in which any self-intersections that arise from the projection have independent tangent directions.  The self-intersections are \textit{crossings}.  Projections of equivalent links can have different numbers of crossings.  The projection plane is understood to have a positive normal direction, allowing us to say that one of the branches at the crossing has a ``higher'' projection preimage; this branch is the \textit{overcrossing strand}, and the other is the \textit{undercrossing strand}.  In figures, the overcrossing strand appears to pass on top of the other.  With this crossing information, the projection is called a \textit{knot (or link) diagram}.

In the diagram of an oriented link, a sign is given to each crossing. If, possibly after a rotation, a neighborhood of the crossing appears as on the left side of Figure \ref{fig:crossing_sign} then the crossing is positive (note the directionality of the arrows). Otherwise, it appears as on the right of the figure and the crossing is negative. The \textit{writhe} of the diagram equals the number of positive crossings minus the number of negative crossings.  The diagram of Figure \ref{fig:knot_projection} has writhe equal to 0.

Given a link diagram of $K$, if the overcrossing and undercrossing strands are interchanged at each crossing, we get the \textit{mirror of $K$}.  This link is equivalent to the image of $K$ under a map that negates one coordinate of $\mathbb R^3$.

The cycle diagrams introduced in Section \ref{sec:intro} are examples of link diagrams; each crossing has a vertical overcrossing strand and a horizontal undercrossing strand.

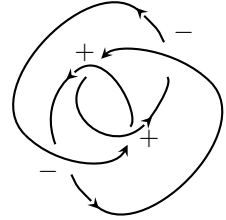
\begin{wrapfigure}{r}{0.38\textwidth}
\centering
    \vspace{-2mm}\begin{tikzpicture}[>=stealth,scale=0.44]
        \draw[thick] (1.9,1.5) ..controls (1.7,1.9) and (1.8, 1.8) ..(1.3,2.3);
        \draw[thick,>-] (1.3,2.3) ..controls (-0.2,3.8) and (-3.8,0.2) ..(-2.3,-1.3);
        \draw[thick,->] (-2.3,-1.3) ..controls (-1.3,-2.3) and (0.4,-2.4) ..(0.8,-1.6);
        \draw[thick] (-0.9,-2.5) ..controls (-0.7,-2.9) and (-0.8, -2.8) ..(-0.3,-3.3);
        \draw[thick,>-] (-0.3,-3.3) ..controls (1.2,-4.8) and (4.8,-1.2) ..(3.3,0.3);
        \draw[thick,->] (3.3,0.3) ..controls (2.3,1.3) and (0.5,1.5) ..(0,1);
        \draw[thick] (-0.45,0.45)    ..controls (-1.45,-0.6) and (0.3,-2) ..(1.3,-1);
        \draw[thick,>-] (1.3,-1)  ..controls (1.45,-0.85) and (2.2,0.20) ..(2,0.45);
        \draw[thick,rotate around={165:(0.2,-0.3)}] (-0.65,0.25)    ..controls (-0.9,-0.15) and (0.7,-2) ..(1.4,-0.9);
        \draw[thick,>-, rotate around={165:(0.2,-0.3)}] (1.4,-0.9) ..controls (1.5,-0.75) and (2.2,0.20) ..(1.4,1.35);
        \draw (-1.6,-2.4) node {$-$}
              (2.5,1.8) node {$-$};
        \draw (1.45,-1.4) node {$+$}
              (-0.5,1.2) node {$+$};
    \end{tikzpicture}
    \vspace{-3mm}
    \caption{An oriented knot diagram with writhe 0}
    \label{fig:knot_projection}
\end{wrapfigure}

As a convention, the orientation of the link associated to a cycle diagram follows the order in which the permutation transitions through the indices.  That is, vertical (resp.\ horizontal) segments above the diagonal are oriented up (resp.\ rightward), and below the diagonal are oriented down (resp.\ leftward). 

Considering the crossings that appear in a cycle diagram and the orientation, each crossing is negative.  As a consequence, if a knot is associated to a cycle diagram, its mirror is in a class of knots called \emph{positive} knots.  Many knots, even among those with a small number of crossings, do not fit into this class (even up to equivalence).  Some obstructions to being positive are known (e.g.\ see \cite{Crom-Homogeneous}). 

% include stuff about links, and about the linking number of two components (like the writhe) and how that is an invariant of link type.
%It is known \cite{Crom} that any knot or link $K$ is equivalent to one represented by a grid diagram. Moreover, and there are a series of rules, called ``Cromwell moves,'' that can transform any grid diagram representation of $K$ into any other grid diagram representing a knot equivalent to $K$. In the realm of grid diagrams, these serve the same role as the widely known Reidemeister moves (for more background, see \cite{Lickorish}).

\subsection{Boundaries of surfaces}\label{subsec:surfaces} Every (oriented) link can be realized as the boundary of a connected, oriented surface in $\mathbb R^3$.  Such a surface is called a \textit{Seifert surface} of the link, after Herbert Seifert who gave an algorithm for producing one, given a link diagram \cite{Seifert} (the existence of a surface was known earlier \cite{Frankl-Pontr}).  The minimal genus of a Seifert surface of $K$ is called the \textit{genus} of $K$, written $g(K)$.  Equivalent links have equal genus; also, the mirror of $K$ will have the same genus as $K$.

A part of Seifert's algorithm, to be used in Section \ref{sec:topology}, involves determining Seifert circles from an oriented link diagram.  To define Seifert circles, consider a crossing (with an orientation, so each strand has an incoming and outgoing end).  Remove the crossing point and connect the coherently oriented strands ("smooth" the crossing), as in Figure \ref{fig:smoothing}.  By smoothing every crossing of the diagram, we obtain a collection of pairwise disjoint circles in the plane (circles in a topological sense).  These are the \textit{Seifert circles} of the diagram.  See an example in Figure \ref{fig:my_label}.

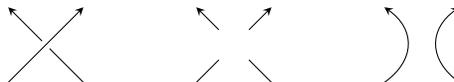
\begin{figure}[ht]
    \begin{tikzpicture}[>=stealth]
        \draw[->]
            (1,0)--(0.55,0.45)
            (0.45,0.55)--(0,1);
        \draw[->]
            (0,0)--(1,1);
        \draw[->]
            (3.5,0)--(3.2,0.3)
            (2.8,0.7)--(2.5,1);
        \draw[->]
            (2.5,0)--(2.8,0.3)
            (3.2,0.7)--(3.5,1);
        \draw (5,0) edge[out=35,in=-35,looseness=1.3,->] (5,1);
        \draw (6,0) edge[out=145,in=215,looseness=1.3,->] (6,1);
    \end{tikzpicture}
    \caption{Smoothing a crossing: remove the crossing, then reconnect the remaining edges in the same orientation without crossing.}
    \label{fig:smoothing}
\end{figure}

\subsection{Legendrian links} Related to the modern study of contact geometry are \textit{Legendrian links}.  Loosely speaking, given a contact structure on $\mathbb R^3$, a knot or link is \textit{Legendrian} if it satisfies a tangency condition based on the contact structure (additionally, equivalence of two Legendrians is constrained by the tangency condition). We refer the interested reader to the survey article \cite{Etn}.
%Focusing on Legendrian knots in $\mathbb R^3$, with the standard contact structure, these are simple closed curves which also satisfy a tangency condition at each point. The condition is that the tangent vector field of the curve is in the kernel of the one-form $dz - ydx$ (where coordinates $(x,y,z)$ are placed on $\mathbb R^3$)\todo{this might be too technical?}. Legendrian knots are considered up to a similar notion of equivalence as knots, however the family of deformations $F_t$, that carry one Legendrian to another, are required to keep the tangency condition on the curve at all times $t$.

In the ``standard'' contact structure, Legendrian links are often studied via a specific projection, the \textit{front projection}.  This projection has some peculiarities: it has no vertical tangencies, but has \textit{cusps} (locally like the cuspidal cubic); also, at a crossing, the more negatively-sloped branch is always the overcrossing strand.  A link diagram having these two properties (and being smooth except at cusps), is the front projection of a (unique) Legendrian link.

A grid diagram determines a knot or link (simply viewing it as a link diagram, with vertical overcrossings).  We will also determine a Legendrian link from a grid diagram as follows.\footnote{Our method of determining a Legendrian from a grid diagram is not typical in the literature \cite{NT}; however, our method is closely related, and convenient for our purposes.} First, rotate the grid diagram 45$^\circ$ clockwise.  Then, turn what were originally lower-left and upper-right corners into cusps, and smooth out the upper-left and lower-right corners (now local extrema vertically).  Finally, at each crossing interchange overcrossings with undercrossings (see Figure \ref{fig:LegFront}; the grid diagram has vertical overcrossing strands, and the front projection has negatively-sloped overcrossing strands).  

Suppose that $D$ is a grid diagram and that, thinking of $D$ as a link diagram, $K(D)$ is the associated link.  Let $\Lambda(D)$ be the Legendrian link of the front projection that we associated to $D$, as above.  Then, as a regular link, $\Lambda(D)$ is equivalent to the mirror of $K(D)$.  Note that if every crossing of $D$ is negative, then every crossing of $\Lambda(D)$ is positive.  The fact that $\Lambda(D)$ is the mirror of $K(D)$, and not generally equivalent to $K(D)$, will not affect the arguments of Section \ref{sec:topology}.

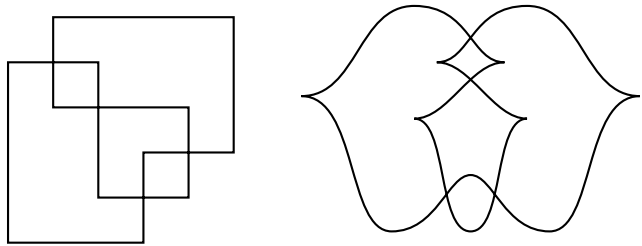
\begin{figure}[ht]
    \[\begin{tikzpicture}[>=stealth,scale=0.15]
        \draw[thick]
            (-26,4)     -- (-26,20) -- (-18,20)
                        -- (-18,8) -- (-10,8)
                        -- (-10,16) -- (-22,16)
                        -- (-22,24) -- (-6,24)
                        -- (-6,12) -- (-14,12)
                        -- (-14,4) -- cycle;
        \draw[thick]
            (10,15) ..controls (13,15) and (12,5) ..(15,5)
                    ..controls (18,5) and (17,15) ..(20,15)
                    ..controls (17,15) and (15,20)..(12,20)
                    ..controls (15,20) and (15,25)..(20,25)
                    ..controls (25,25) and (25,17)..(30,17)
                    ..controls (25,17) and (26,5) ..(22,5)
                    ..controls (18,5) and (17,10) ..(15,10)
                    ..controls (13,10) and (12,5) ..(8,5)
                    ..controls (4,5) and (5,17) ..(0,17)
                    ..controls (5,17) and (5,25)..(10,25)
                    ..controls (15,25) and (15,20) ..(18,20)
                    ..controls (15,20) and (13,15) ..(10,15);
    \end{tikzpicture}\]
\caption{A grid diagram (left) and its front projection (right)}
\label{fig:LegFront}
\end{figure}

An invariant of Legendrian knots and (oriented) links that is of interest to us is the Thurston-Bennequin number. Given a Legendrian $\Lambda$, let $F_{\Lambda}$ be its front projection. The \textit{Thurston-Bennequin} number $\operatorname{tb}(\Lambda)$ is equal to the writhe of $F_{\Lambda}$ minus one-half the number of cusps:
\[\operatorname{tb}(\Lambda) = \text{writhe}(F_{\Lambda}) - \frac12(\#\text{cusps}(F_{\Lambda})).\]
In Section \ref{sec:topology} we will need the Bennequin-Eliashberg inequality \cite{Eliash}, which says that if $\Lambda$ is equivalent, as a regular knot or link, to some $K$, then 
    \[\operatorname{tb}(\Lambda) \leq 2 g(K) - 1.\]
    
\section{Signed Trees} \label{sec:trees}
Various authors \cites{BBL,SS} introduce \textit{separating trees} to study separable permutations.  A separating tree is a rooted binary tree in which each internal node is designated as either positive or negative, and then trees are divided into equivalence classes under certain allowable tree rotation operations. (See also \cite{BR}.)  Using separating trees as motivation, we define a similar structure, which will be useful in the enumeration of unknotted cycles. 

\forestset{
  tria/.style={
    node format={
      \noexpand\node [
      draw,
      shape=regular polygon,
      regular polygon sides=3,
      inner sep=0pt,
      outer sep=0pt,
      \forestoption{node options},
      anchor=\forestoption{anchor}
      ]
      (\forestoption{name}) {\foresteoption{content format}};
    },
    child anchor=parent,
  },
}%
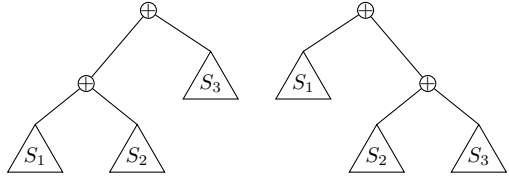
\begin{figure}[h!]
    \centering
    \begin{forest}
for tree={scale=0.8}
[, phantom, s sep = 0.5cm
[$+$,circle,draw,inner sep=-2.2mm,minimum size=1.2mm
  [$+$,circle,draw,inner sep=-2.2mm,minimum size=1.2mm
    [$S_1$, tria]
    [,phantom]
    [$S_2$, tria]
  ]
  [,phantom]
  [,phantom]
  [$S_3$, tria]
]
[$+$,circle,draw,inner sep=-2.2mm,minimum size=1.2mm
  [$S_1$, tria]
  [,phantom]
  [,phantom]
  [$+$,circle,draw,inner sep=-2.2mm,minimum size=1.2mm
    [$S_2$, tria]
    [,phantom]
    [$S_3$, tria]
  ]
]
]
\end{forest}
    \caption{The result of a single tree rotation.  The triangles $S_i$ represent additional segments of the tree whose relative position changes in the rotation, but internally they remain unchanged.}
     \label{fig:TreeRotation}
\end{figure}

\begin{definition} \label{def:RSBT}
A \textbf{rooted-signed-binary tree} is a rooted binary tree in which each node except the root is given a sign, positive or negative.  Furthermore, we say two binary rooted trees are  equivalent if one can be obtained from another by a series of tree rotations.  A tree rotation (see Figure \ref{fig:TreeRotation} for an example) is allowed at a given node if either:
\begin{enumerate}
    \item The parent node has the same sign as the child rotating into its place.
    \item The node is the root node, in which case the newly created node is assigned the same sign as the node that was rotated into the position of the root. 
\end{enumerate}
\end{definition}

For example, the seven trees depicted in Figure \ref{fig:exampletree} are all equivalent, and represent all of the allowed rotations of the given tree.

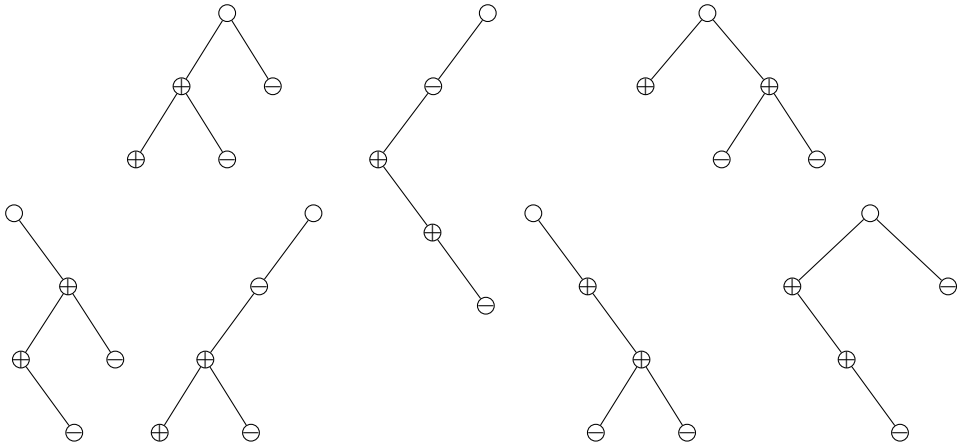
\begin{figure}[h!]
    \centering
    \begin{forest}
for tree={scale=0.95,circle,draw,inner sep=-2.1mm,minimum size=1.1mm}
[, phantom, s sep = 1cm
[\phantom{+} 
  [$+$
    [$+$]
    [,phantom]
    [$-$]
  ]
  [$-$
    [,phantom]
  ]
]
[\phantom{+} 
  [$-$
     [$+$
       [,phantom]
       [,phantom]
       [+
          [,phantom]
          [,phantom]
          [$-$]
       ]
     ]
     [,phantom]
     [,phantom]
  ]
  [,phantom]
  [,phantom]
]
[\phantom{+} 
  [$+$
    [,phantom
       [,phantom]
       [,phantom]
    ]
    [,phantom]
  ]
  [$+$
    [$-$]
    [,phantom]
    [$-$
       [,phantom]
       [,phantom]
       [,phantom]
    ]
  ]
]
]
\end{forest}
\vspace{-15mm}

\begin{forest}
for tree={scale=0.95,circle,draw,inner sep=-2.1mm,minimum size=1.1mm}
[, phantom, s sep = 0.9cm
[\phantom{+} 
  [,phantom]
  [,phantom]
  [$+$
    [$+$
       [,phantom]
       [,phantom]
       [$-$]
    ]
    [,phantom]
    [$-$]
  ]
]
[\phantom{+} 
  [$-$
     [$+$
       [$+$
          [,phantom]
          [,phantom]
       ]
       [,phantom]
       [$-$]
     ]
     [,phantom]
     [,phantom]
  ]
  [,phantom]
  [,phantom]
]
[\phantom{+} 
  [,phantom]
  [,phantom]
  [$+$
    [,phantom]
    [,phantom]
    [$+$
      [$-$]
      [,phantom]
      [$-$]
    ]
  ]
]
[\phantom{+} 
  [$+$
    [,phantom]
    [,phantom]
    [+
       [,phantom]
       [,phantom]
       [$-$]
    ]
  ]
  [,phantom]
  [,phantom]
  [$-$
    [,phantom]
    [,phantom]
  ]
]
]
\end{forest}
    \caption{All seven unique tree rotations of a rooted-signed-binary tree.}
     \label{fig:exampletree}
\end{figure}

As our rooted-signed-binary-trees differ slightly from the separating trees which are known to be in bijection with separable permutations, we give a proof that the number of such trees with $n$ nodes is counted by the Schr\"oder Numbers.

\begin{prop}
The number of rooted-signed-binary-trees with $n$ nodes is counted by the $n$-th Schr\"oder number, $S_n$.
\end{prop}
\begin{proof}
We show this using generating functions.  Let $T(x)$ be the ordinary generating function for such trees, and take as our representative for each equivalence class of rooted-signed-binary-trees the tree in which all possible left rotations have been made.  Note that such a tree has a root node with no right child.  

There are three possibilities for the left child: no left child, a left child which itself has no right child, or a left child whose right child has the opposite sign.  These three cases are counted by $x$, $2xT(x)$, and $2xT(x)\left(\frac{T(x)}{2x}-\frac{1}{2}\right)=T(x)^2-xT(x)$ respectively.  The last two require brief explanations.  Trees where the left child of the root has no right child can be constructed by taking any tree (counted by $T(x)$), assigning the root node either of 2 signs and then making that node the left child of a new signless root node, (increasing the size by 1) giving $2xT(x)$.  

Now, if the left child of the root has a right child of the opposite sign, the possibilities for that right child can be counted by $\frac{T(x)/x}{2}-1$. The $T(x)/x$ counts the trees with the unlabelled root node removed, dividing by 2 accounts for the fact that the sign of this right child must be the opposite of the node above it, and subtracting one eliminates the possibility of this being empty.  Adding these three cases together and simplifying we obtain 
\[T(x)=x+xT(x)+T(x)^2,\]
the same recurrence as the generating function of the Schr\"oder numbers.
\end{proof}

In the next section we establish a bijection between these trees and the unknotted cycles.  In doing so we will frequently build these trees by inserting (or removing) leaf nodes into existing trees.  If $T$ is a rooted-signed-binary tree and $v$ is a (non-root) leaf of that tree, then we denote by $T-\{v\}$ the tree obtained by removing the node $v$.  Sometimes we may also remove multiple nodes (and write $T-\{v,\ldots\}$), however we will never remove an internal node without also removing its descendants.  

We enumerate the positions that a new leaf-node could be inserted from left to right and refer to them as follows.

\begin{definition}
Suppose a tree $T$ has $n$ nodes, and $v$ is a leaf of $T$.  We say that $v$ is in \textbf{relative position} $i$, where $1\leq i \leq n$, if there are exactly $i-1$ places where a leaf could be inserted in $T$ to the left of $v$ (not counting the left child of $v$), we also say that a node is inserted in position $i$, meaning that after insertion the node is in relative position $i$. 
\end{definition}

Note that the relative position of a leaf is unaffected by any tree rotation of the tree so long as it remains a leaf after the rotation.

%\subsection{Notation}
%We record here some notation and definitions that will be used throughout the remainder of the paper.  

\section{The bijection} \label{sec:bijection}

In this section we describe how to construct an unknotted cycle from a fixed rooted-signed-binary-tree.  We define our construction by describing how each node added to the tree affects the associated cycle.  We start with the rooted-signed-binary-tree consisting only of the unlabelled root, which corresponds to the trivial cycle, $21$ which is clearly unknotted.

\begin{wrapfigure}{l}{0.3\textwidth}
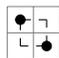

\centering
    \centering
    \cyclefig{2,1}
    \caption{The cycle 21.}
    \label{fig:21}
\end{wrapfigure}

At any point in this construction our tree will have as many places where a new node can be added as the length of the permutation thus far constructed (which is one more than the number of nodes in the tree).  

When the tree consists of only the root node, there are two places where a node can be added, as either left or right children of the root node, corresponding to positions 1 and 2 of the trivial cycle respectively.  

Now, when a new node is added a new point is inserted into the diagram, and other points are shifted up and to the right. We define functions     \begin{equation}
    \xi_m(k) = \begin{cases}
        k \quad & k < m \\
        k+1 \quad & k \geq m.
    \end{cases} \label{eq:x-shift}
\end{equation}
    When a node is added to the tree in relative position $i$, it affects the corresponding cycle $\sigma=s_1s_2\ldots s_n$    (in one line notation) according to the following rules:
\begin{enumerate}
    \item If a positive node is added, $i+1$ is inserted into $\sigma$ prior to the element in position $i$.  Each element having value $i+1$ or greater is increased by one.
    \[s_1s_2\ldots s_n \ \rightarrow \ \xi_{i+1}(s_1)\xi_{i+1}(s_2)\ldots \xi_{i+1}(s_{i-1})\ [i{+}1]\ \xi_{i+1}(s_{i})\ldots \xi_{i+1}(s_{n})\]
    \item If a negative node is added, $i$ is inserted into $\sigma$ after the element in position $i$.  Each element having value $i$ or greater is increased by one.
    \[s_1s_2\ldots s_n \ \rightarrow \ \xi_{i}(s_1)\xi_{i}(s_2)\ldots \xi_{i}(s_{i})\ [i]\ \xi_{i}(s_{i+1})\ldots \xi_{i}(s_{n})\]
\end{enumerate}

In terms of the cycle diagrams, this has the effect of taking one of the corners where the diagram made a right angle at the $y=x$ line and changing it into a notch or a kink with a new off-diagonal point in the cycle diagram.  The exact change depends on the behavior of the lines in the diagram prior to the insertion, and are summarized in Table \ref{tab:insertions}.
\begin{table}
    \centering
    \begin{tabular}{|c|c|c|}
    \hline
    Before insertion & After inserting $\oplus$ & After inserting $\ominus$ \\
    \hline
    
      \begin{tikzpicture}[scale=0.7]
      \draw[opacity=0] (0,1.1);
      \draw [-Latex] (-1,0)--(0,0)->(0,1);
      \draw[dotted] (-0.5,-0.5)--(0.5,0.5);
\end{tikzpicture}   &        \begin{tikzpicture}[scale=0.7]
      \draw[opacity=0] (0,1.1);
      \draw[fill] (-0.5,0) circle (0.07);
      \draw [-Latex] (-1,-0.5)--(-0.5,-0.5)--(-0.5,0)--(0,0)->(0,0.7);
      \draw[dotted] (-1,-1)--(0.6,0.6);
\end{tikzpicture}  &      \begin{tikzpicture}[scale=0.7]
      \draw[opacity=0] (0,1.1);
      \draw[fill] (0,-0.5) circle (0.07);
      \draw [-Latex] (-1.3,0)--(0,0)--(0,-0.5)--(-0.5,-0.5)->(-0.5,0.8);
      \draw[dotted] (-0.8,-0.8)--(0.5,0.5);
\end{tikzpicture}\\

      \begin{tikzpicture}[scale=0.7]
      \draw[opacity=0] (0,0.6);
      \draw [-Latex] (-0.8,0.3)--(0.3,0.3)->(0.3,-0.7);
      \draw[dotted] (-0.5,-0.5)--(0.5,0.5);
\end{tikzpicture}   &        \begin{tikzpicture}[scale=0.7]
      \draw[opacity=0] (0,0.6);
      \draw[fill] (-0.5,0) circle (0.07);
      \draw [-Latex] (-1,-0.5)--(-0.5,-0.5)--(-0.5,0)--(0,0)->(0,-1.2);
      \draw[dotted] (-1,-1)--(0.3,0.3);
\end{tikzpicture}  &      \begin{tikzpicture}[scale=0.7]
      \draw[opacity=0] (0,0.6);
      \draw[fill] (0,-0.5) circle (0.07);
      \draw [-Latex] (-1.3,0)--(0,0)--(0,-0.5)--(-0.5,-0.5)->(-0.5,-1.3);
      \draw[dotted] (-1.2,-1.2)--(0.3,0.3);
\end{tikzpicture}\\

      \begin{tikzpicture}[scale=0.7]
      \draw[opacity=0] (0,1.5);
      \draw [-Latex] (0.7,-0.3)--(-0.3,-0.3)->(-0.3,0.7);
      \draw[dotted] (-0.5,-0.5)--(0.5,0.5);
\end{tikzpicture}   &        \begin{tikzpicture}[scale=0.7]
      \draw[opacity=0] (0,1.5);
      \draw[fill] (0,0.5) circle (0.07);
      \draw [-Latex] (1.3,0)--(0,0)--(0,0.5)--(0.5,0.5)->(0.5,1.3);
      \draw[dotted] (1.2,1.2)--(-0.3,-0.3);
\end{tikzpicture}  &      \begin{tikzpicture}[scale=0.7]
      \draw[opacity=0] (0,1.5);
      \draw[fill] (0.5,0) circle (0.07);
      \draw [-Latex] (1.3,0.5)--(0.5,0.5)--(0.5,0)--(0,0)->(0,1.3);
      \draw[dotted] (1.2,1.2)--(-0.3,-0.3);
\end{tikzpicture}\\

      \begin{tikzpicture}[scale=0.7]
      \draw[opacity=0] (0,-1.3);
      \draw [-Latex] (1,0)--(0,0)->(0,-1);
      \draw[dotted] (-0.5,-0.5)--(0.5,0.5);
\end{tikzpicture}   &        \begin{tikzpicture}[scale=0.7]
      \draw[opacity=0] (0,-1.2);
      \draw[fill] (0,0.5) circle (0.07);
      \draw [-Latex] (1.3,0)--(0,0)--(0,0.5)--(0.5,0.5)->(0.5,-0.9);
      \draw[dotted] (0.8,0.8)--(-0.5,-0.5);
\end{tikzpicture}  &      \begin{tikzpicture}[scale=0.7]
      \draw[opacity=0] (0,-1);
      \draw[opacity=0] (0,1.4);
      \draw[fill] (0.5,0) circle (0.07);
      \draw [-Latex] (1,0.5)--(0.5,0.5)--(0.5,0)--(0,0)->(0,-0.7);
      \draw[dotted] (1,1)--(-0.6,-0.6);
\end{tikzpicture}\\
\hline
\end{tabular}

    \caption{A graphical depiction of how the cycle diagram changes when a node is added to the tree.  The first column depicts the corner on the diagonal before the new point is inserted.  Note, when a $\oplus$ node is inserted, a corner is created above the diagonal, and conversely, if a $\ominus$ is inserted, the new corner is below the diagonal. }    \label{tab:insertions} \vspace{-1.5em} \end{table}
\raggedbottom

From these pictures it is clear visually that these changes will not affect the knot of the cycle diagram, up to equivalence.  (Formally, each of the changes is either planar isotopy or a Reidemeister I move.) Thus if a given cycle is unknotted before one of these operations is performed, it will still correspond to the unknot afterward.  The proof of the following proposition follows immediately.

\begin{prop}
Inserting an off-diagonal element to a permutation $\sigma$ (an $i+1$ in position $i$ or an $i$ in position $i+1$) and shifting the points above or to the right in the cycle diagram results in a permutation associated to the same link as $\sigma$.
\end{prop}

By repeated application of this proposition we get the following theorem.

\begin{theorem}
If a cycle $\sigma$ is obtained from a rooted-signed-binary-tree by the construction above (processing nodes in some order) then $\sigma$ is an unknotted cycle. 
\end{theorem}

We illustrate this by building the cycle corresponding to the tree in Figure \ref{fig:exampletree}. We use the first diagram depicted in that figure. The reader is invited to verify the same cycle is obtained for any equivalent tree and irrespective of the order the nodes of the tree are considered, as we subsequently  prove.

\begin{example}\label{example1}
We consider the nodes from the first tree of Figure \ref{fig:exampletree} one at a time.  Starting with the root node, we have the trivial cycle 21, shown in Figure \ref{fig:21}.  

We first process the positive, left child of the root.  This positive node is in relative position 1, so we insert 2 into position 1, obtaining 231.  We could also have found this from the cycle diagram, noting the corner in position 1 was a lower left corner (row 3 of table \ref{tab:insertions}) and replacing the corner in the diagram with the picture in the second column.  The cycle 231 is is depicted first in figure \ref{fig:24531}.

\begin{figure}[h!]
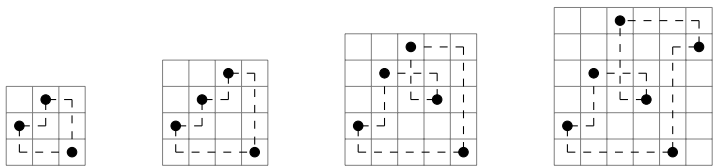

    \centering
    \cyclefig{2,3,1} \hspace{0.5cm} \cyclefig{2,3,4,1} \hspace{0.5cm} \cyclefig{2,4,5,3,1} \hspace{0.5cm} \cyclefig{2,4,6,3,1,5}
    \caption{The cycles 231, 2341, 24531 and 246315.}
    \label{fig:24531}
\end{figure}

Now consider the leftmost leaf.  It again is positive, in relative position 1, so we get 2341. Then continue to the negative node to the right of the node just considered.  The two potential children of the previous, positive, node corresponded to positions 1 and 2 of the cycle, so this node occupies position 3.  As it is negative, we insert a 3 after position 3 of the cycle, obtaining 24531, depicted third in figure \ref{fig:24531}.

Last we consider the right, negative child of the root.  It is in relative position 5, so we insert 5 after the last position of the cycle, obtaining our final cycle, 246315.  

Note in this case all 3 of the elements of the permutation on the off diagonal, the elements 2,3 and 5 (in positions 1,4 and 6 respectively) correspond to leaf nodes. This is not the case for all representations of the tree, after certain rotations only some of these off-diagonal points still correspond to leaves.

\end{example}

From the description it isn't immediately clear this construction produces the same cycle regardless of the order in which nodes are processed, even before potential tree rotations are taken into account.  We will show that the bijection is well-defined on a fixed tree diagram. First, to reduce the number of cases we must consider later, we demonstrate a relationship between the cycle constructed from a rooted-signed-binary tree and the one obtained from the ``negative'' of that tree.
% maybe should be consistent on always calling it a cycle, or saying permutation

\begin{lemma} 
Let $T$ be a rooted-signed-binary tree and $\overline{T}$ the rooted-signed-binary tree obtained from $T$ by reversing the sign of each node. If $\sigma$ is the cycle produced by processing the nodes of $T$ in some fixed order, then processing the (oppositely signed) nodes of $\overline{T}$ in the same order produces $\sigma^{-1}$ the functional inverse of $\sigma$.
\label{lemma:inverse-cycle}
\end{lemma}

\begin{proof}
The statement is certainly true when $T$ consists of only the root node. As the root is unsigned, $T = \overline{T}$ and the associated cycle $21$ is its own inverse. 

Before continuing, in the cycle diagram of $\sigma$ suppose that for indices $i,j$ there is a consecutive vertical-horizontal pair of segments between diagonal points $(i,i)$ and $(j,j)$.  By definition, if the segments are above the diagonal then $\sigma(i) = j$; if they are below the diagonal then $\sigma(j)=i$.  It follows that the cycle diagram obtained by reflecting the cycle diagram of $\sigma$ across $y=x$ corresponds to the permutation $\sigma^{-1}$.

Now, suppose the statement holds for any rooted-signed-binary tree of size $n$. Let $T$ be a tree of size $n+1$, and $v$ a leaf of $T$.  Let $\sigma$ be the permutation obtained from processing the nodes of $T-\{v\}$ in some order. By induction, $\overline{T-\{v\}}$ corresponds to $\sigma^{-1}$ (processing the oppositely-signed nodes in the same order) and the corresponding cycle diagram is the reflection across the diagonal.

Exchanging the role of $T$ and $\overline{T}$ if necessary, assume $v$ is positive and inserted in position $i$.  Thinking of the cycle diagram of $\sigma$ as determined by its $n$ non-diagonal points $(i,\sigma(i))$, the effect of inserting $v$ into $T-\{v\}$ is to move all points with height greater than $i$ up one position, and shift all points in horizontal position at least $i$ to the right by one. Additionally, a point at $(i,i+1)$ is inserted.  On the other hand, inserting a negative vertex in position $i$ of $\overline{T-\{v\}}$ shifts points in the diagram of $\sigma^{-1}$ that are in horizontal position greater than $i$ to the right, and moves points with height at least $i$ up one position. Additionally, a point at $(i+1,i)$ is added.

We know that points which were unchanged, from the diagram corresponding to $T-\{v\}$ to the one for $T$, reflect across the diagonal to points of the diagram for $\overline{T-\{v\}}$.  The observations above explain why those points which were shifted (or added) upon insertion of $v$ will reflect across the diagonal to those which shift (or are added) when $\overline{v}$ is inserted.  Thus the permutations corresponding to $T$ and $\overline{T}$ (processing nodes in that fixed order) are functional inverses.  \end{proof}

\begin{prop}
The cycle produced by applying the above construction to the nodes of a fixed representation of a rooted-signed-binary tree is the same, regardless of the order in which the nodes of the tree are processed.
\label{prop:process-independent}
\end{prop}

In order to prove Proposition \ref{prop:process-independent} we will first prove the following.
\begin{lemma}
The cycle produced by the above construction is not changed when the order of processing the last two leaves of the tree changes.
\label{lemma:process-twoleaves}
\end{lemma}

\begin{proof}
Suppose $T$ is a rooted-signed-binary tree with at least two leaves of with $n+2$ positions where a new leaf could be inserted. Fix two distinct leaves, $v$ and $w$ of $T$ in relative positions $i$ and $j$ respectively of $T-\{v,w\}$ where $i<j$.

By processing the nodes of $T-\{v,w\}$ in some order, our construction assigns a cycle $\sigma$ to $T-\{v,w\}$. Write $\sigma = s_1s_2...s_n$.  We show $v$ and $w$ can be added to the tree in either order, and the corresponding changes to $\sigma$ produce the same cycle.  

First note we can assume by Lemma \ref{lemma:inverse-cycle} that $v$ is signed negatively (if it weren't we could swap $T$ for $\overline{T}$ and show that the inverse cycle can be constructed unambiguously regardless of the order in which the leaves are added).  Recall the function $\xi_m$ defined in \eqref{eq:x-shift}.  For any two integers $l<m$ it is easily verified that 
\begin{equation}
    \xi_l(\xi_m(k))=\xi_{m+1}(\xi_l(k)) \label{eq:swap-x}
\end{equation} holds for all integers $k$.  We consider two cases, based on the sign of $w$.  

{\bf Case 1:} $w$ is positive.  Inserting the negatively-signed $v$ first, in position $i$ of $T-\{v,w\}$, we transform the associated permutation to \begin{equation}
    s_1s_2...s_n \ \rightarrow \ \xi_{i}(s_1)\xi_{i}(s_2)\ldots \xi_{i}(s_{i})\ [i]\ \xi_{i}(s_{i+1})\ldots \xi_{i}(s_{n}). \label{eq:vfirst}
\end{equation}  
Since $w$ is to the right of $v$, $w$ is now in relative position $j+1$ of $T-\{w\}$.  Also $\xi_i(s_j)$ is the $(j+1)$-element in (\ref{eq:vfirst}), so inserting the positively signed $w$ results in 
\begin{align}&\ldots \xi_{j+2}(\xi_{i}(s_{i}))\xi_{j+2}(i)\xi_{j+2}(\xi_{i}(s_{i+1}))\ldots \xi_{j+2}(\xi_{i}(s_{j-1}))\ [j{+}2]\ \xi_{j+2}(\xi_{i}(s_{j}))\ldots \nonumber \\
&=\ldots \xi_{i}(\xi_{j+1}(s_{i}))\ [i]\ \xi_{i}(\xi_{j+1}(s_{i+1}))\ldots \xi_{i}(\xi_{j+1}(s_{j-1}))\ [j{+}2]\ \xi_{i}(\xi_{j+1}(s_{j}))\ldots \label{eq:positive-w}
\end{align}
where the second line was obtained using \eqref{eq:swap-x} to interchange the two functions. On the other hand, if we first add $w$ to $T-\{v,w\}$, the first permutation obtained is \[s_1s_2...s_n \ \rightarrow \ \xi_{j+1}(s_1)\xi_{j+1}(s_2)\ldots \xi_{j+1}(s_{j-1})\ [j+1]\ \xi_{j+1}(s_{j})\ldots \xi_{j+1}(s_{n}).\] Inserting the negatively signed $v$ into position $i$ of $T-\{v\}$ is then seen to immediately yield \eqref{eq:positive-w} since $\xi_i(j+1)=j+2$. 

{\bf Case 2:} $w$ is negative.  As in case 1, if we first insert the negatively-signed node $v$ we obtain \eqref{eq:vfirst}.  Again, $w$ is to the right of $v$ in $T$, so $w$ is now in relative position $j+1$ of $T-\{w\}$. Inserting the negatively signed $w$ in position $j+1$ results in 
\begin{align}&\ldots \xi_{j+1}(\xi_{i}(s_{i}))\xi_{j+1}(i)\xi_{j+1}(\xi_{i}(s_{i+1}))\ldots \xi_{j+1}(\xi_{i}(s_{j}))\ [j{+}1]\ \xi_{j+1}(\xi_{i}(s_{j+1}))\ldots \nonumber \\
&=\ldots \xi_{i}(\xi_{j}(s_{i}))\ [i]\ \xi_{i}(\xi_{j}(s_{i+1}))\ldots \xi_{i}(\xi_{j}(s_{j}))\ [j{+}1]\ \xi_{i}(\xi_{j}(s_{j+1}))\ldots \label{eq:negative-w}
\end{align}
interchanging the two functions using \eqref{eq:swap-x}.  If, instead, we first add $w$ to $T-\{v,w\}$, the first, intermediate, permutation obtained is \[s_1s_2...s_n \ \rightarrow \ \xi_{j}(s_1)\xi_{j}(s_2)\ldots \xi_{j}(s_{j})\ [j]\ \xi_{j}(s_{j+1})\ldots \xi_{j}(s_{n}).\] Now, inserting the negatively signed $v$ into position $i$ of $T-\{v\}$ again transforms this to \eqref{eq:negative-w} after noting that $\xi_i(j)=j+1$. 
\end{proof}

\begin{proof}[Proof of Proposition \ref{prop:process-independent}]
Suppose for contradiction there exist trees for which the cycle construction described is ambiguous, and let $T$ be a minimal sized counterexample.  Let $v$ and $w$ be the last nodes of $T$ processed in two different orders resulting in different cycles.  Note that $v$ and $w$ must both be leaves of $T$, and that $v \neq w$, otherwise the construction of $T-\{v\}$ would also be ambiguous. 

Because $T$ is a minimal counterexample, any order of processing the nodes of $T-\{v\}$ must result in the same cycle, so we can assume $w$ is the last node inserted into $T-\{v\}$.  Likewise, we may assume in the order where $w$ is last that $v$ is second-to-last. Thus, each ordering first constructs a cycle for $T-\{v,w\}$, and in each case the cycles constructed on $T-\{v,w\}$ must agree. Since we supposed that these orders resulted in different cycles (when processing $v$ then $w$ versus $w$ then $v$), we have contradicted the statement of Lemma \ref{lemma:process-twoleaves}.
%Supposing that there are $n$ positions in the tree $T-\{u,v\}$, let $v$ be added to $T-\{u,v\}$ at position $j$ of that tree, and then let $u$ be added to $T-\{u\}$ at position $i$. 
%
%Assume that $i\le j$. If $\sigma$ is the permutation associated to $T-\{u,v\}$, then this order of adding $v$ then $u$ will produce the permutation     %\[\rp{n+1}{i}\left(\rp{n}{j}\sigma'\tp{n}{j}\right)'\tp{n+1}{i}.\]
%
%We need to show that this agrees with the associated permutation if $u$ is added first and then $v$. Since $i\le j$, that permutation would be
%\[\rp{n+1}{j+1}\left(\rp{n}{i}\sigma'\tp{n}{i}\right)'\tp{n+1}{j+1}.\]
%
%The fact that these permutations agree is a direct consequence of Lemma \ref{lem:fixed-compatible} and the observation that two pairs that are $\sigma''$-fixed compatible produce the same function under the appropriate composition.
%
%Now, note that the case when $i> j$ is taken care of by simply beginning with the second order of adding the nodes $u$ and $v$ and relabeling the positions. Since the construction of the permutation is unambiguous on $T-\{u,v\}$, and we have shown that the order that $u$ and $v$ are added is inessential, once the permutation on that smaller tree is known, we have contradicted our choice of $T$.
\end{proof}

\begin{lemma} \label{lem:tree-rotation-leaf}
Suppose $\sigma$ is a cycle of length at least 3 obtained from a rooted-signed-binary tree $T$. If $\sigma(i)-i=1$, there is a tree rotation of $T$ with a positively signed leaf in relative position $i$ or, if $\sigma(i)-i=-1$, a negatively signed leaf in relative position $i-1$. 
\end{lemma}

\begin{proof}
Note that if $\sigma(i)-i=-1$ then $\sigma^{-1}(i-1)-(i-1)=1$. By Lemma \ref{lemma:inverse-cycle} we can reduce to the case $\sigma(i)-i=1$ by considering the inverse cycle with the oppositely signed tree. Hence, we consider only the case $\sigma(i)-i=1$. 

The cycles of length 3 are 231 and 312. Only the former has $\sigma(i)-i=1$ (for $i=1,2$).  In this case $T$ has one positive leaf, which can be rotated into either position 1 or 2.  Now, suppose the claim holds for all cycles of length $n-1$. Let $\sigma$ have length $n$, $\sigma(i)-i=1$,  $T$ a tree corresponding to $\sigma$ with a leaf $w$ in relative position $k$. If $k=i$ and $w$ were positive we are done.  If $k=i$ with $w$ negative then $\sigma(i+1)=i$. But $\sigma(i)=i+1$ by assumption, and since $\sigma$ is a cycle we would be left with $\sigma=21$. This is not the case so we can assume $k\ne i$.

Let $\sigma^*$ be the cycle corresponding to $T-\{w\}$. If $w$ is negative and $k=i+1$, then $\sigma(i+2)=i+1=\sigma(i)$ which is impossible. Thus it is straightforward to check that if $k>i$ then $\sigma^*(i)=\sigma(i) = i+1$ and, respectively, if $k<i$ that $\sigma^*(i-1)=\sigma(i)-1 = i$, regardless of the sign of $w$. Since $\sigma^*$ has length $n-1$, there exists a rotation of $T-\{w\}$, call it $(T-\{w\})'$, containing a positive leaf $u$ in relative position $i$ (respectively $i-1$) of $(T-\{w\})'$.  Note the children of $u$ would be in relative positions $i$ and $i+1$ (respectively $i-1$ and $i$).

If either $k<i-1$ or $k>i+1$ then $u$ would still be a leaf after $w$ is inserted into position $k$ of $(T-\{w\})'$ and, in either case, $u$ will be in position $i$.  Since tree rotations do not affect the relative position of any leaf (so long as the rotation doesn't cause the node to no longer be a leaf)
%\todo{Should this observation be made earlier? And ``uninvolved'' means ``wasn't the node that rotated through, but could have been a descendant of that node'' right?},
applying the same rotations to $T$ required to transform $T-\{w\}$ into $(T-\{w\})'$ will result in a tree having a leaf in the desired position.

If $k=i-1$ and $w$ were negative, then $\sigma(i)=i-1$ which isn't the case, so if $k=i-1$ or $k=i+1$ then $w$ is positive, and a child of the leaf $u$ in $T'$.  In either case we can perform a tree rotation of $w$ into $u$ producing a positively signed vertex in position $i$ as desired, and the same argument as above applies. \qedhere

\end{proof}

\begin{prop}
The construction of cycles from rooted-signed-binary trees is injective, trees that aren't related by tree rotations don't produce the same cycle.
\end{prop}

\begin{proof}
This is clear if $\sigma=21$.  Suppose for contradiction that there exists a minimal unknotted cycle $\sigma$ obtained from distinct rooted-signed-binary trees $T$ and $T'$, unrelated by tree rotations.  Let $v$ be a leaf of $T$, in relative position $i$.  It follows from the construction that $\sigma(i)-i = \pm 1$.  By Lemma \ref{lem:tree-rotation-leaf} there exists a tree rotation of $T'$ so that $T'$ also has a leaf $w$ in position $i$, with the sign as $v$.  

Let $\sigma^*$ be the cycle obtained from $T-\{v\}$.  Since inserting $v$ in position $i$ of $T-\{v\}$ results in $\sigma$, as does inserting $w$ in position $i$ of $T'-\{w\}$, we find, by reversing the insertion rules, that $\sigma^*$ is necessarily the cycle obtained from $T'-\{w\}$ as well. 

By the minimality of $\sigma$,  since $T-\{v\}$ and $T'-\{w\}$ are both associated to $\sigma^*$, there exists a sequence of tree rotations to transform $T-\{v\}$ to $T'-\{w\}$.  None of these rotations involves the leaf $v$  since it is not present, so the same sequence of rotations can be performed on $T$.  Performing this sequence of rotations to $T$ will leave the leaf $v$ unchanged in position $i$, while transforming the remainder of the tree, $T-\{v\}$, into $T-\{w\}$.  Thus this sequence results in the leaf $v$ moving into the same position occupied by $w$ in $T'$ and thus transforms $T$ into $T'$, a contradiction. \end{proof}

So far we have not considered the allowed tree rotations of a rooted-signed-binary tree.  In this section we show that two such trees which can be obtained from one another by allowed tree rotation moves correspond to the same cycle.

\begin{theorem}
Any rooted-signed-binary-trees $T$ and $T'$ related by the rotations of Definition \ref{def:RSBT} produce the same cycle under the construction of Section \ref{sec:bijection}.
\end{theorem}
\begin{proof}
It suffices to show that any single tree rotation results in the same cycle.  The result holds vacuously if $T$ consists only of the unsigned root node.  Suppose now the result holds for all trees smaller than $T$, and suppose $T'$ is related to $T$ by a tree rotation at some vertex $v$. By swapping $T$ and $T'$ if necessary, we suppose this rotation is a ``clockwise'' rotation, which moves the left child of $v$ into the position of $v$, with $v$ becoming the new right child of that vertex in $T'$.  

Note, $v$ and its left child must have the same sign.  If this sign is negative we  can replace $T$ and $T'$ with $\overline{T}$ and $\overline{T'}$ respectively so (by Lemma \ref{lemma:inverse-cycle}) we can suppose $v$ and its left child are both positive ($v$ could be the root, which is unsigned, but we  still assume that its left child is positive, which doesn't change the argument that follows.)  We can draw the relevant portion of $T$ and $T'$ (showing only the vertex $v$ and its descendants) as in Figure \ref{fig:TreeRotation}, where $S_1$, $S_2$ and $S_3$ are the relevant subtrees for descendants of $v$ and its left child. 

Suppose $w$ is a leaf appearing in one of these subtrees.  Since $w$ is not involved in the rotation it is a leaf of both $T$ and $T'$.  The trees $T-\{w\}$ and $T'-\{w\}$ are related by the same tree rotation at $v$ and, since both are smaller than $T$ both correspond to the same cycle $\sigma$, by induction.  Since $w$ appears in the same relative position in both $T$ and $T'$, adding $w$ to either of $T-\{w\}$ or $T'-\{w\}$ would produce the same change to $\sigma$, thus $T$ and $T'$ correspond to the same cycle.

This leaves us with the case where all the subtrees $S_1$, $S_2$ and $S_3$ are empty. Then the vertex $v$ in $T$ has a single positive left child $w$ (with no children) while in $T'$ the vertex $w$ only has a right child, $v$. Note that $T-\{w\}$ is the exact same tree as $T'-\{v\}$, let $i$ be the relative position of $v$ in $T-\{w\}$ and let $\sigma$ be the cycle associated to it.   Since $v$ is a positive leaf of $T-\{w\}$ we can write \[\sigma = s_1s_2,\ldots s_{i-1}\ [i+1]\ s_{i+1}\ldots s_n\]
where $s_j \neq i+1$.  Now, the left and right children of $v$ correspond to the positions $i$ and $i+1$ respectively.  We check what occurs to the cycle $\sigma$ when a positive node is inserted in either place.  If a positive left child is inserted (creating $T$) then our rules dictate that an $i+1$ is inserted in position $i$ of $\sigma$, and all elements of sigma with value $i+1$ or higher are increased by 1.  Thus this insertion results in
\begin{align}
    \xi_{i+1}&(s_1)\xi_{i+1}(s_2)\ldots \xi_{i+1}(s_{i-1})\ [i+1]\ \xi_{i+1}(i+1)\xi_{i+1}(s_{i+1})\ldots \xi_{i+1}(s_n) \nonumber \\
    &=\xi_{i+1}(s_1)\xi_{i+1}(s_2)\ldots \xi_{i+1}(s_{i-1})\ [i+1][i+2]\ \xi_{i+1}(s_{i+1})\ldots \xi_{i+1}(s_n). \label{eq:leftch}
\end{align}
If a positive vertex is inserted as the right child (position $i+1$) of $v$ (creating $T'$) then our rules dictate that an $i+2$ is inserted in position $i+1$ of $\sigma$, and all elements of sigma with value $i+2$ or higher is increased by 1.  Thus this insertion results in 
\begin{align}
    \xi_{i+2}&(s_1)\xi_{i+2}(s_2)\ldots \xi_{i+2}(s_{i-1})\xi_{i+2}(i+1)\ [i+2]\ \xi_{i+2}(s_{i+1})\ldots \xi_{i+2}(s_n) \nonumber \\
    &=\xi_{i+2}(s_1)\xi_{i+2}(s_2)\ldots \xi_{i+2}(s_{i-1})\ [i+1][i+2]\ \xi_{i+2}(s_{i+1})\ldots \xi_{i+2}(s_n). \label{eq:rightch}
\end{align}
Since $\xi_{i+1}(s_j) = \xi_{i+2}(s_j)$ so long as $s_j \neq i+1$, we see that \eqref{eq:leftch} and \eqref{eq:rightch} are the same expression, and so $T$ and $T'$ correspond again to the same cycle. \end{proof}

%\todo[inline]{Prove construction well defined under tree-rotations.  Reduce to the case of a parent and child of the same sign with no additional children.  Show equivalent under rotation, and then later insertions produce same diagrams}

\label{sec:welldef}

\section{Surjectivity} \label{sec:topology}
In this section we prove that the map from rooted-signed-binary trees to unknotted cycles is surjective. The proof relies on Bennequin's inequality \cite{Benn}, or more precisely a reformulation of it for Legendrian knots \cite{Eliash} {--} an important early result in modern contact geometry.  We begin with some notation.

\begin{definition} \label{def:Cpair-URindex}
Let $\sigma$ be a derangement. Define $(i, j)$ to be a \emph{$C$-pair} if either
    \begin{align*}
        i < j < \sigma(i) < \sigma(j) \qquad\text{or}\qquad
        i > j > \sigma(i) > \sigma(j).
    \end{align*}
Additionally, define $i$ to be a \emph{$UR$-index} if 
        \[\sigma^{-1}(i) < i \qquad\text{and}\qquad \sigma(i) < i.\]
\end{definition}
\begin{remark}
Note that $(i, j)$ is a $C$-pair if and only if two segments of the cycle diagram cross each other, one of the segments occurring between diagonal points $(i,i)$ and $(\sigma(i), \sigma(i))$ and the other occurring between $(j,j)$ and $(\sigma(j), \sigma(j))$.  Such a pair appears as depicted in the left-most image of Figure \ref{fig:C-UR} (or its reflection across the diagonal). All crossings in cycle diagrams are of this type. 

Also, $i$ is a $UR$-index if and only if the cycle diagram has an upper right corner at $(i,i)$, as in the right-most image of Figure \ref{fig:C-UR}.
\end{remark}

\begin{figure}[ht]
    \centering
    \begin{tikzpicture}[scale=0.3]
        \draw[help lines] (-0.5,-0.5) grid (5.5, 5.5);
        \draw[dashed]   (0.5,0.5) -- (0.5,3.5) -- (3.5,3.5)
                        (1.5,1.5) -- (1.5,4.5) -- (4.5,4.5);
        \draw   (0.5,3.5) [fill] circle (.18)
                (1.5,4.5) [fill] circle (.18);
        \draw[help lines] (8.5,-0.5) grid (11.5,2.5);
        \draw[dashed]   (8.5,1.5) -- (10.5,1.5) -- (10.5,-0.5);
    \end{tikzpicture}
    \caption{The cycle diagram at a $C$-pair (left) and $UR$-index (right).}
    \label{fig:C-UR}
\end{figure}
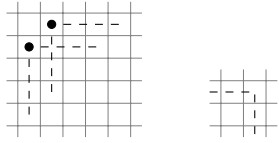

\begin{prop}[Bennequin-Eliashberg Inequality] \label{prop:Benn-Inequality}
Let $D$ be a cycle diagram and let $K(D)$ be the knot associated to $D$. Define $C(D)$ to be the number of $C$-pairs of $D$ and $UR(D)$ the number of $UR$-indices. Recalling that $g(K)$ is the genus of $K$, then 
        \[C(D) - UR(D) \le 2g(K) - 1.\]
\end{prop}
\begin{proof}
Let $\Lambda(D)$ be the Legendrian defined in Section \ref{sec:knots_intro}.  Cycle diagrams have all negative crossings, and so recall that each crossing of the front projection of $\Lambda(D)$ is positive, and so its writhe is $C(D)$.  Points on $D$ at a $UR$-index correspond to right cusps, and so
        \[C(D) - UR(D) = \text{tb}(\Lambda(D)).\]
Recalling that the mirror of $K(D)$ is equivalent to $\Lambda(D)$ (simply as links), and that genus does not change under mirroring, the Bennequin-Eliashberg inequality gives $\text{tb}(\Lambda(D)) \le 2g(K) - 1$.
\end{proof}

\begin{remark}A tighter result may be obtained by using Lagrangian fillings and results in \cites{Chan,HS}, but it is not necessary for our purposes.
%In fact, for our setting we know have a stronger result than Bennequin's inequality. Since we have a Legendrian knot $\Lambda$ (more precisely, a `front projection' of $\Lambda$) where all crossing are positive, we know that there is an exact Lagrangian surface that fills $\Lambda$ in the standard symplectic structure on $\mathbb R^4$ \cite{HS}. Any Legendrian with such a filling has rotation number 0 and Thurston-Bennequin number equal to $2g_s(K) - 1$ \cite{Chan} (here $g_s$ is the minimal genus of a smooth surface in $\mathbb R^4$ which has boundary $\Lambda$). Thus, when the smooth genus equals the Seifert genus (for example, when $K$ is an unknot and $g(K) = 0$), we have equality
%            \[C(D_\sigma) - UR(D_\sigma) = 2g(K) - 1.\]
\end{remark}

\begin{prop} \label{lem:C-UR}
Suppose that $\sigma$ is an $n$-cycle where $|\sigma(i)-i|\geq 2$ for all $1\leq i \leq n$. Then the number of $C$-pairs of $\sigma$ is at least as many as the number of $UR$-indices. 
\end{prop}

Let $D_\sigma$ be the cycle diagram of $\sigma$. Our proof uses the Seifert circles of $D_\sigma$, introduced in Section \ref{subsec:surfaces}. Recall the orientation convention for $D_\sigma$ (found in Section \ref{sec:knots_intro}).  To each Seifert circle $S$ associate a set $\mathcal C(S)$ of crossings of $D_\sigma$ {---} the crossings which, when smoothed, created a portion of the Seifert circle. 

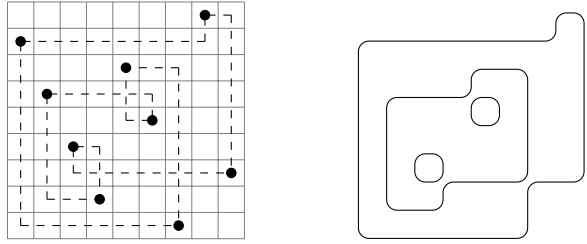
\begin{figure}[ht]
    \centering
    \cyclefig{8,6,4,2,7,5,1,9,3} 
    \hspace{1cm} \begin{tikzpicture}[scale=3.75]
    \draw [rounded corners] (0,0)--(0,0.7)--(0.7,0.7)--(0.7,0.8)--(0.8,0.8)--(0.8,0.2)--(0.6,0.2)--(0.6,0.0)--cycle;
    \draw [rounded corners] (0.1,0.1)--(0.1,0.5)--(0.4,0.5)--(0.4,0.6)--(0.6,0.6)--(0.6,0.2)--(0.3,0.2)--(0.3,0.1)--cycle;
    \draw [rounded corners] (0.2,0.2)--(0.2,0.3)--(0.3,0.3)--(0.3,0.2)--cycle;
    \draw [rounded corners] (0.4,0.4)--(0.4,0.5)--(0.5,0.5)--(0.5,0.4)--cycle;
\end{tikzpicture}
    \caption{The cycle diagram of the unknotted cycle 864275193 and the corresponding Seifert circles obtained by smoothing the three crossings.}
    \label{fig:my_label}
\end{figure}

To prove the proposition, we will use few lemmas.
\begin{lemma}\label{lem:UR=n_SeifertCircles}
The number of Seifert circles of $D_\sigma$ is equal to the number of $UR$-indices of $D_\sigma$.
\end{lemma}
\begin{proof}
First, make a key observation: due to the orientation on $D_\sigma$, smoothing a crossing cannot \emph{create} a local extremum of the function $x+y$ on a Seifert circle. Therefore, any such local extremum must correspond to a point on $D_\sigma$ itself.  So it is a diagonal point, at a $UR$-index if it is a local maximum.
%In terms of $C$-pairs, the crossing of a $C$-pair $(i,j)$ is associated to the Seifert circle that passes through diagonal points $j$ and $\sigma(i)$ and also to the Seifert circle passing through diagonal points $i$ and $\sigma(j)$. The orientation of the strands prevents the smoothing of that crossing from connecting diagonal points $i$ and $j$.

As a Seifert circle is a closed planar curve, it has at least one local maximum of $x+y$, and so passes through at least one diagonal point that is at a $UR$-index. 

The Seifert circles of $D_\sigma$ have an induced orientation that agrees with the orientation on $D_\sigma$ away from the crossings (oriented clockwise in the case of cycle diagrams).  If $S$ is a Seifert circle of $D_\sigma$, then at each point where $S$ passes through the diagonal, it goes from above the diagonal to below it. Hence, $S$ cannot pass through two $UR$-indices since, were it to do so, it would spiral either inward or outward and be unable to form a simple closed curve. Thus, every Seifert circle passes through exactly one $UR$-index.
\end{proof}
\begin{remark}
Similarly, there is exactly one lower-left corner on the diagonal (an $LL$-index, say) contained in any Seifert circle of a cycle diagram.
\end{remark} 

Now, place a partial ordering on the set of Seifert circles. If $S$ and $S'$ are Seifert circles, then we say that $S' \prec S$ when $S'$ is contained in the bounded planar region that is enclosed by $S$.

\begin{lemma}\label{lem:unique-maximal}
There is a unique maximal Seifert circle of $D_\sigma$, larger than every other Seifert circle.
\end{lemma}
\begin{proof} Let $S$ be a maximal Seifert circle (in some maximal chain). If there exist Seifert circles in the unbounded region of $S$, then at least one such circle $S'$ must share a crossing with $S$.  Otherwise it would not be a knot ($\sigma$ would not be a cycle). 

Now, suppose that the $UR$-index for $S$ is less than the $UR$-index for $S'$. Since $S\not\prec S'$, the $UR$-index of $S$ is also less than the unique $LL$-index of $S'$. However, this makes it impossible for there to be a crossing in $\mathcal C(S) \cap \mathcal C(S')$. Likewise, if we suppose that the $UR$-index of $S$ is greater than that of $S'$, then the $LL$-index of $S$ must also be larger, making a shared crossing impossible. 

Hence, every other Seifert circle is in the bounded region of $S$.
\end{proof}

\begin{lemma}\label{lem:number-crossings-Seifert-circle}
If $D_\sigma$ has a crossing, then any Seifert circle $S$ has at least 1 associated crossing. Also, under the hypothesis of Lemma \ref{lem:C-UR}, if $|\mathcal C(S)|=1$, then $S$ is the maximal Seifert circle.
\end{lemma}
\begin{proof} The first statement is clear, since otherwise $S$ is a closed curve already in the (unsmoothed) cycle diagram and, as the diagram has a crossing elsewhere, there would need be multiple components (a link, rather than a knot). 

Suppose $\mathcal C(S)$ contains only one crossing. It must be that either $S$ is a minimal element in the partial order, or it is maximal. If this weren't so, then there would be a Seifert circle in both the bounded and unbounded regions of $S$. But the fact that the knot has a single component then necessitates another crossing in $\mathcal C(S)$. 

Now, suppose $S$ is minimal. Let $\ell$ be the $LL$-index of $S$ and let $r$ be the $UR$-index (which is unique by Lemma \ref{lem:UR=n_SeifertCircles}). Since $S$ is minimal, $S$ must intersect every diagonal point between $\ell$ and $r$.  We can divide $S$ into segments, each having one vertical part and one horizontal part.  Each segment will start and end on the diagonal: for some $i$ and $j$ the segment has a vertical segment from $(i,i)$ to $(i,j)$ followed by a horizontal segment to $(j,j)$.  Call such a segment a \emph{step} of $S$, with length $|i-j|$. 

If any step of length 1 occurs in $S$, from $(i,i)$ to $(i+1,i+1)$ for some $i$, then there is a smoothed crossing that was located at $(i,i+1)$ (the step could not have existed in $D_\sigma$ by the hypothesis of Lemma \ref{lem:C-UR}).  So if we show that $S$ has at least two steps of length 1, then we have a contradiction. Note that this is automatic if $r - \ell = 1$ (with one step above the diagonal and one below), so we assume that $r - \ell > 1$.

For $\ell<i<r$, the point $(i,i)$ is part of two steps of $S$, either both above the diagonal or both below. Now consider the point $(\ell+1,\ell+1)$. As it is part of two steps (and is not an $LL$-index), it must be connected to $(\ell,\ell)$, a step of length 1. A similar argument holds for $r-1$ and $r$, giving another step of length 1. These steps are not the same as $r-\ell > 1$. By the above argument, $|\mathcal C(S)|>1$, a contradiction.

Therefore, $S$ cannot be minimal in the partial order when $|\mathcal C(S)|=1$, and $S$ is the (unique) maximal Seifert circle.
\end{proof}

\begin{proof}[Proof of Proposition \ref{lem:C-UR}]
By Lemma \ref{lem:UR=n_SeifertCircles}, we must show that the the number of Seifert circles of $D_\sigma$ is at most the number of crossings.  Let $s$ denote the number of such Seifert circles.  Since each crossing is associated to two Seifert circles, and by Lemma \ref{lem:number-crossings-Seifert-circle} every Seifert circle $S$ has $|\mathcal C(S)|\ge2$ (except possibly the maximal Seifert circle), the total number of crossings must be at least $\left\lceil \frac{2s-1}{2} \right\rceil = s$.
\end{proof}

If $K$ is an unknot, then it bounds a disk and so the Seifert genus is $g(K) = 0$. Hence we obtain the following.

\begin{prop} \label{prop:unknot-has-off-diagonal}
If the link associated to the cycle diagram for $\sigma$ is the (single component) unknot, then there is some $i$ such that $|i - \sigma(i)| = 1$.
\end{prop}
\begin{proof}
Let $D_\sigma$ be the cycle diagram. Since $g(K) = 0$, Proposition \ref{prop:Benn-Inequality} implies that $C(D_\sigma) \le UR(D_\sigma) - 1$. By Lemma \ref{lem:C-UR} we must have some $i$ such that $|i - \sigma(i)| < 2$. The fact that $\sigma$ is a derangement thus implies the conclusion.
\end{proof}
\begin{prop}
Every unknotted cycle is obtained from a rooted-signed-binary tree through the construction of Section \ref{sec:bijection}.
\label{prop:surjective}
\end{prop}
\begin{proof}
The proof will be by induction on the length of the cycle. A cycle of length 2 it must be $\sigma=21$ which corresponds to the tree containing only the root.

Let $D$ be the cycle diagram of an unknotted cycle $\sigma$ of length $n$. By Proposition \ref{prop:unknot-has-off-diagonal}, there is an $i$, with $1\le i\le n$ so that $|i - \sigma(i)| = 1$. If $i-\sigma(i)=1$ then, near the $i^{th}$ diagonal point, $D$ must be like the depiction in one of the rows in the third column of Table \ref{tab:insertions}.  Let $D_0$ be the cycle diagram obtained by collapsing the $i$ and $i-1$-st rows and columns of $D$, replacing them with what appears in the same row, but first column of \ref{tab:insertions}. (Thus, $D_0$ is a cycle diagram with one fewer row and column than $D$.) If $i-\sigma(i)=-1$ instead, then $D$ must be like the depiction in one of the rows of the second column of Table \ref{tab:insertions}. In similar manner to the previous case, define $D_0$ by using the same row, but in the first column.

As it has fewer diagonal nodes, we may assume that $D_0$ is obtained from a rooted-signed-binary tree $T_0$. Now, define a tree $T$ by adding a node to $T_0$: if $i-\sigma(i)=1$ then add a negative node to $T_0$ at position $i-1$; if $i-\sigma(i)=-1$ then add a positive node to $T_0$ at position $i$. Then clearly $T$ is assigned cycle diagram $D$ by the construction.
\end{proof}

\section{Further Results} \label{sec:links}
Extending these ideas, we can count all unlinked permutations (derangements).  By an unlink, we mean the knot corresponding to each cycle of the permutation is an unknot in the grid diagram, and furthermore each of these unknots are not ``linked'' with one another.  We obtain a bivariate generating function that keeps track of both the size of the permutation and the number of knots in the link.  

\begin{theorem} \label{thm:unlinkedderangements}
Let $\mathcal{U}$ be the set of unlinked derangements and denote by $cyc(\sigma)$ the number of cycles in $\sigma$ (equivalently knots in the link associated to $\sigma$).
Define the bivariate generating function \begin{equation}
    F(u,x) = 1+\sum_{\sigma \in \mathcal{U}} u^{cyc(\sigma)}x^{|\sigma|}  \label{eq:unlinkedGF}
\end{equation}
to count unlinked permutations by length and number of components.  Then $F(u,x)$ satisfies the recurrence
\begin{equation} 2 + (ux-2)F(u,x) -ux^2F(u,x)^2-uxF(u,x)\sqrt{1-6xF(u,x)+x^2F(u,x)^2})=0 \label{eq:linkrecur} \end{equation}
or, equivalently, \[ 1 + (u x-2)F(u,x) + (1 - u x - u x^2)F(u,x)^2 + (u x^2 + u^2 x^3)F(u,x)^3=0.\]
\end{theorem}

Setting $u{=}1$ we recover the generating function $f(x)$ for the sequence of all unlinked permutations, and find that $f(x)$ satisfies the recurrence \[ 1 + (-2+ x)f(x) + (1 - x - x^2)f(x)^2 + (x^2 + x^3)f(x)^3=0\] corresponding to the sequence 1, 2, 8, 32, 143, 674, 3316, 16832, 87538$\ldots$.  This, as well as the sequences counting unlinked permutations with $k$ disjoint cycles for each $k<5$ are listed in Table \ref{tab:unlinked}.  Note that the number of unlinked permutations of $2n$ with $n$ components is the $n$-th Catalan number, which is easy to prove by relating the nested links to Dyck paths.

    \begin{table}[h]    
    \centering
\begin{tabular}{|c|c|c|}
\hline
$k$ &$\#$ $k$-component unlinked permutations& Generating Function \\
\hline
1     & $0,1,2,6,22,90,394,1806,8558,41586\ldots$ & $\frac{1}{2} \left(x-x^2-x\sqrt{x^2-6 x+1} \right)$\\
2     & $0,0,0,2,10,48,238,1216,6354,33760\ldots$ & $\frac{1}{2}x^2{-}3x^3{+}x^4{-}\frac{x^2-9 x^3 +12 x^4-2 x^5}{2\sqrt{x^2-6 x+1}}$ \\
3     & $0,0,0,0,\ \hspace{0.6mm}0,\ \hspace{0.6mm}5,\ \hspace{0.6mm}42,\ \hspace{0.5mm}280,1752,10710\ldots$ & $\ldots$ \\
4     & $0,0,0,0,\ \hspace{0.6mm}0,\ \hspace{0.6mm}0,\ \hspace{1.1mm}\ 0,\ \hspace{1mm}\ 14,\ \hspace{0.8mm}168,\ \hspace{0.5mm}1440\ldots$ & $\ldots$ \\
\hline
all & $0,1, 2, 8, 32, 143, 674, 3316, 16832,87538\ldots$ & $f(x)$ \\
\hline
\end{tabular}
    \caption{The number of unlinked permutations by the number of components in the link. \label{tab:unlinked}} \vspace{-1.5em}
    %The generating function for 3 component unlinks is given by \\     $    f_3(x) = \frac{-5 x^{10}+80 x^9-442 x^8+965 x^7-713 x^6+202 x^5-24 x^4+x^3}{2  \left(x^2-6 x+1\right)^2}$ \\ \hspace*{1cm} $+\frac{-5 x^9+65 x^8-267 x^7+366    x^6-143 x^5+21 x^4- x^3}{2   \left(x^2-6 x+1\right)^{3/2}}    $  }
\end{table}
To prove Theorem \ref{thm:unlinkedderangements} we need the following result.
\begin{lemma} \label{lem:noncrossing}
Suppose $\sigma$ is an unlinked permutation, with cycle decomposition $\sigma=\sigma_1\sigma_2\cdots\sigma_k$.  Then each of the component cycles $\sigma_i$ are unknotted cycles, and the cycle diagrams of any two cycles are noncrossing.
\end{lemma}
\begin{proof}
We show that the cycle diagrams of two different component cycles of an unlinked permutation cannot cross each other by considering the linking number of the corresponding unknots.  Let $K$ and $K'$ be two different components of a link $L$, and consider a diagram in the plane of $L$. Define $c_+$ (resp.\ $c_-$) to be the number of positive (resp.\ negative) crossings in the diagram, where we only count crossings that involve one strand from $K$ and one strand from $K'$. The \emph{linking number} of $K$ and $K'$ equals $\nicefrac{1}{2}(c_+ - c_-)$.

For any planar diagram that represents a link equivalent to $L$ there will be two components corresponding to $K$ and $K'$. It is well-known (see e.g. \cite[pg. 11]{Lickorish}) that the linking number (in that diagram) of those two components must be equal to the previously computed number, $\nicefrac{1}{2}(c_+ - c_-)$. That is, the linking number of $K$ and $K'$ is invariant under change of diagram. (In fact, one can verify this by checking that the linking number will be unchanged by Reidemeister moves.)

Since there is a diagram of the $cyc(\sigma)$ component unlink which has no crossings at all, the linking number of any two components in any diagram must be zero. However, all crossings are negative in cycle diagrams. If two component cycle diagrams cross at any point, then the linking number of those components cannot be zero, and so the diagram cannot be that of an unlinked permutation.  
\end{proof}

Define the support of a length $n$ permutation $\sigma$, $\sup(\sigma)=\{1\leq i \leq n \mid \sigma(i) \neq i \}$.  
\begin{lemma} \label{lem:linkdecomp}
Suppose $\sigma$ is a derangement of $n$ whose cycle diagram contains no crossings between different components of its cycle diagram.  Let $\sigma'$ denote the cycle of $\sigma$ containing 1 and let $1=\sigma'_1 < \sigma'_2< \cdots <\sigma'_j$ be the elements of $\sup(\sigma')$.  Then there exist permutations $\tau_1$, $\tau_2$, \ldots, $\tau_j$ with \begin{equation} \sigma = \sigma'\circ \tau_1\circ \tau_2\circ \cdots\circ \tau_j \label{eq:sigmadecomp} \end{equation} such that the support of  $\tau_i$ is precisely the integers between $\sigma'_i$ and $\sigma'_{i+1}$,
\begin{equation}
    \sup(\tau_i)=\{k|\sigma'_i<k<\sigma'_{i+1}\},\text{ for } i<j \text{ and } \sup(\tau_j)=\{k|\sigma'_j<k \leq n\}.
\end{equation}
\end{lemma}

\begin{wrapfigure}{r}{0.41\textwidth}
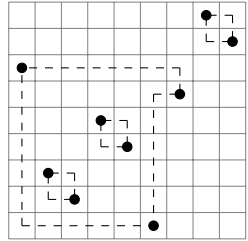

    \centering
    \cyclefig{7,3,2,5,4,1,6,9,8}
    \caption{$\sigma = 732541698$}
    \label{fig:732541698}
\end{wrapfigure}

\
\vspace{-7mm}

\begin{example}
Take $\sigma=732541698$ (one line notation) as depicted in Figure \ref{fig:732541698}.  Written in cycle notation $\sigma=(176)(23)(45)(89)$.  Then (in cycle notation) $\sigma'=(176)$, ${\sigma'_1=1}, {\sigma'_2=6}, \sigma'_3=7$ and $\tau_1=(23)(45)$, $\tau_2$ is the identity, and $\tau_3=(89)$.
\end{example}

\begin{proof}[Proof of Lemma \ref{lem:linkdecomp}]
For each $1\leq i\leq j$, define $\tau_i$ to be the permutation consisting of all cycles from $\sigma$ that contain in their support at least one element strictly between $\sigma_i'$ and $\sigma_{i+1}'$ (or strictly greater than $\sigma_j'$, in the case of $\tau_j$).  It suffices to show that every element in $\sup(\tau_i)$ is contained in the desired range, between $\sigma'_i$ and $\sigma'_{i+1}$ (or greater than $\sigma'_j$, when $i=j$).  

To this end, suppose to the contrary there were a $k < \sigma'_{i+1}$ with $\tau_i(k)>\sigma'_{i+1}$.  (The case where $\tau_i(k)<\sigma'_{i}$ proceeds similarly.)  While $\tau_i$ need not consist of a single cycle, there must exist some $k'\in\sup(\tau_i)$ so that $k'> \sigma'_{i+1}$, and yet $1 <\tau_i(k')<\sigma'_{i+1}$ (the first iteration of $\tau_i$, $n\geq 1$, so that $k'=\tau_i^n(k) > \sigma'_{i+1}$ but $\tau_i^{n+1}(k) < \sigma'_{i+1}$).  

Recall that if $k < \sigma(k)$, the corresponding vertical and horizontal segments of the cycle diagram for $\sigma$ must be \emph{above} the diagonal, while if $k > \sigma(k)$, the corresponding segments must be \emph{below} the diagonal. Thus, the elements $k, k'$ determine a cycle component whose corresponding diagram surrounds the diagonal point $(\sigma'_{i+1},\sigma'_{i+1})$. That component's diagram is a closed curve in the plane, as depicted in Figure \ref{fig:lem_cycle_supports}.  Since elements in the support of $\tau_i$ must all be greater than $1$, the point $(1,1)$ lies in the unbounded region for this closed curve, and the diagram for $\sigma'$ must reach the bounded region to arrive at $(\sigma'_{i+1}, \sigma'_{i+1})$. There must then be a crossing between these components, contradicting the assumption there is no such crossing. \qedhere
\end{proof}

\begin{wrapfigure}{l}{0.38\textwidth} \centering 
    \begin{tikzpicture}[scale=0.45]
        \draw[help lines] (-0.5,-0.5) grid (7.5, 7.5);
        \draw[thin, dashed]   (0.5,0.5) -- (0.5,5.5) -- (5.5,5.5) -- (5.5,3.5) -- (3.5,3.5);
        \draw[very thick,dashed, blue!50!black]   (2.5,2.5) -- (2.5,4.5) -- (4.5,4.5);
        \draw[very thick, dashed, blue!50!black]  (6.5,6.5) -- (6.5,1.5) -- (1.5,1.5);
        \draw[thin, dashed, blue!50!white]    (4.5,4.5) ..controls (3.9,5.1) and (3.9,5.3) ..(4.5,5.5)
        ..controls (5.1,5.7) and (5.2,5.8) ..(5.2,6.2)
        ..controls (5.2,6.6) and (6.3,7.1) .. (6.5,6.5);
        \draw[thin, dashed, blue!50!white]    (2.5,2.5) ..controls (1.7,2.5) and (1.5,2.3) .. (1.5,1.5);
        \draw[fill=white]   (3.5,3.5) circle (.18);
        \draw[fill=gray]    (0.5,5.5) circle (.18)
                            (5.5,3.5) circle (.18);
        \foreach \p in {(6.5,1.5), (2.5,4.5)}
        \draw[fill=blue!50!black]   \p circle (.18);
        \draw (4.2,2.9) node {\footnotesize $\sigma'_{i+1}$};
        \draw   (2.7,2.3) node {\footnotesize $k$}
                (6.9,6.5) node {\footnotesize $k'$};
    \end{tikzpicture}
    \caption{The cycle diagram lines for $\tau_i$.}
    \label{fig:lem_cycle_supports}
\end{wrapfigure}

Finally, we count unlinked permutations.

\begin{proof}[Proof of Theorem \ref{thm:unlinkedderangements}]

In the definition of the bivariate generating function $F(u,x)$ the initial term 1 accounts for an empty derangement, which will be useful in the recursion.  For the remainder of the proof we assume we are counting nontrivial derangements, which necessarily have length at least 2.

We now construct all such unlinked derangements as follows. We select first the cycle containing the element 1.  Once we fix the number of number of elements to be in this cycle, $k$, then by Theorem \ref{thm:main} the number of possible choices for a cycle supported on $k$ elements is $S_{k-1}$. (Note that the knot type is unchanged when the numbers in the support of the cycle are changed, so long as their relative order is preserved, as shifting the values, without changing their order has the effect only of changing the lengths of the lines in the corresponding cycle diagram.)

Having chosen the number of elements permuted by the cycle containing 1, $\sigma'$ as well as the relative cycle type on those elements, we can apply Lemma \ref{lem:linkdecomp}, which tells us that in between each of the elements of $\sigma'$ we can insert any unlink (including potentially the empty unlink) which is counted by our original generating function $F(u,x)$.  Since we can insert such an unlink between any of the $k$ elements of $\sigma'$, as well as after the last element, there are a total of $k$ places where such an unlink can be inserted.  Adding together these possibilities, we have
\begin{align*}
    F(u,x) = 1 + \sum_{k=2}^\infty S_{k-1}ux^kF(u,x)^k &= 1+ uxF(u,x) \sum_{k=1}^\infty S_{k}\big(xF(u,x)\big)^{k} \nonumber \\
    &= 1+ ux F(u,x) S\big(xF(u,x)\big).
\end{align*}

Using that $S(x)=\frac{1}{2}\left(1-x-\sqrt{1-6x+x^2}\right)$ now gives 

\begin{align*}
    F(u,x)    &= 1+ \frac{ux}{2}F(u,x) \big(1-xF(u,x)-\sqrt{1-6xF(u,x)+(xF(u,x))^2})\big)
\end{align*}
which simplifies to \eqref{eq:linkrecur}. \end{proof}

Finally, one could instead choose to consider all permutations, rather than just derangements, with the convention that any fixed points in the permutation correspond to infinitesmal unknot components of the unlink.  In this case, we obtain the following modification of Theorem \ref{thm:unlinkedderangements} by the same argument. 
\begin{theorem}
Let $\mathcal{V}$ be the set permutations whose cycle diagram corresponds to an unlink (treating fixed points as their own component of an unknot) and define the bivariate generating function \begin{equation}
    G(u,x) = 1+\sum_{\sigma \in \mathcal{V}} u^{cyc(\sigma)}x^{|\sigma|}.  \label{eq:unlinkedpermGF}
\end{equation} Then $G(u,x)$ satisfies the recurrence \begin{equation} 2 + (3ux-2)G(u,x) -ux^2G(u,x)^2-uxG(u,x)\sqrt{1-6xG(u,x)+x^2G(u,x)^2}=0, \label{eq:linkpermrecur} \end{equation} or equivalently \[ux^2G(u,x)^3 + (2u^2x^2 - ux^2 - 3ux  + 1)G(u,x)^2 + (3ux - 2)G(u,x) + 1=0\] 
\end{theorem}

Setting $u=1$ the sequence counting such permutations with any number of components begins \[1, 2, 6, 23, 103, 511, 2719, 15205, 88197,\ldots\]
which appears to match entry \href{http://oeis.org/A301897/}{A301897} in the OEIS.  This sequence counts permutations with the following property.  Given a permutation $\sigma$, let $\mathrm{inv}(\sigma)$ be the number of inversions, $\mathrm{cyc}(\sigma)$ the number of cycles, and $\mathrm{td}(\sigma)$ the total displacement, defined by Diaconis and Graham \cite{diaconis} to be $$\mathrm{td}(\sigma)=\sum_{i=1}^{|\sigma|} |\sigma(i)-i|$$  (see also \cite{peterson}).  Diaconis and Graham prove that $\mathrm{inv}(\sigma)+(|\sigma|-\mathrm{cyc}(\sigma))\leq \mathrm{td}(\sigma)$. 

The OEIS sequence above counts those permutations for which the Diaconis-Graham inequality is an equality.  Jacob Alderink, Samuel Johnson, Noah Jones, Matthew Mills, and Alexander Woo conjecture that this set of permutations is, in fact, precisely the set of permutations giving unlinks.

\bibliographystyle{amsplain}
\bibliography{unknotbib}

\begin{bibdiv}
\begin{biblist}

\bib{Benn}{article}{
      author={Bennequin, D.},
       title={Entrelacement et {\'e}quations de {P}faff},
        date={1983},
     journal={Asterisque},
      volume={107--108},
       pages={83\ndash 161},
}

\bib{BBL}{article}{
      author={Bose, Prosenjit},
      author={Buss, Jonathan~F.},
      author={Lubiw, Anna},
       title={Pattern matching for permutations},
        date={1998},
        ISSN={0020-0190},
     journal={Inform. Process. Lett.},
      volume={65},
      number={5},
       pages={277\ndash 283},
         url={https://doi.org/10.1016/S0020-0190(97)00209-3},
      review={\MR{1620935}},
}

\bib{BR}{article}{
      author={Bouvel, Mathilde},
      author={Rossin, Dominique},
       title={The longest common pattern problem for two permutations},
        date={2006},
        ISSN={1218-4586},
     journal={Pure Math. Appl. (PU.M.A.)},
      volume={17},
      number={1-2},
       pages={55\ndash 69},
      review={\MR{2380348}},
}

\bib{Chan}{article}{
      author={Chantraine, Baptiste},
       title={Lagrangian concordance of legendrian knots},
        date={2010},
     journal={Algebr. Geom. Topol.},
      volume={10},
      number={1},
       pages={63\ndash 85},
         url={https://doi.org/10.2140/agt.2010.10.63},
}

\bib{Crom-Homogeneous}{article}{
      author={Cromwell, P.~R.},
       title={Homogeneous links},
        date={1989},
        ISSN={0024-6107},
     journal={J. London Math. Soc. (2)},
      volume={39},
      number={3},
       pages={535\ndash 552},
  url={https://doi-org.proxy-tu.researchport.umd.edu/10.1112/jlms/s2-39.3.535},
      review={\MR{1002465}},
}

\bib{Crom}{article}{
      author={Cromwell, Peter~R.},
       title={Embedding knots and links in an open book. {I}. {B}asic
  properties},
        date={1995},
        ISSN={0166-8641},
     journal={Topology Appl.},
      volume={64},
      number={1},
       pages={37\ndash 58},
         url={https://doi.org/10.1016/0166-8641(94)00087-J},
      review={\MR{1339757}},
}

\bib{diaconis}{article}{
      author={Diaconis, Persi},
      author={Graham, R.~L.},
       title={Spearman's footrule as a measure of disarray},
        date={1977},
        ISSN={0035-9246},
     journal={J. Roy. Statist. Soc. Ser. B},
      volume={39},
      number={2},
       pages={262\ndash 268},
      review={\MR{652736}},
}

\bib{Dyn}{article}{
      author={Dynnikov, I.~A.},
       title={Arc-presentations of links: monotonic simplification},
        date={2006},
        ISSN={0016-2736},
     journal={Fund. Math.},
      volume={190},
       pages={29\ndash 76},
         url={https://doi.org/10.4064/fm190-0-3},
      review={\MR{2232855}},
}

\bib{Eliash}{article}{
      author={Eliashberg, Yakov},
       title={Contact 3-manifolds twenty years since {J}.\ {M}artinet's work},
    language={eng},
        date={1992},
     journal={Annales de l'institut Fourier},
      volume={42},
      number={1-2},
       pages={165\ndash 192},
         url={http://eudml.org/doc/74949},
}

\bib{Elizalde}{article}{
      author={Elizalde, Sergi},
       title={The {X}-class and almost-increasing permutations},
        date={2011},
        ISSN={0218-0006},
     journal={Ann. Comb.},
      volume={15},
      number={1},
       pages={51\ndash 68},
         url={https://doi.org/10.1007/s00026-011-0082-9},
      review={\MR{2785755}},
}

\bib{Etn}{incollection}{
      author={Etnyre, John~B.},
       title={Legendrian and transversal knots},
        date={2005},
   booktitle={Handbook of knot theory},
   publisher={Elsevier B. V., Amsterdam},
       pages={105\ndash 185},
         url={http://dx.doi.org/10.1016/B978-044451452-3/50004-6},
      review={\MR{2179261 (2006j:57050)}},
}

\bib{Zohar}{article}{
      author={Even-{Z}ohar, C.},
       title={Models of random knots},
        date={2017},
     journal={J. Appl. and Comput. Topology},
      volume={1},
       pages={263\ndash 296},
         url={https://doi.org/10.1007/s41468-017-0007-8},
}

\bib{Frankl-Pontr}{article}{
      author={Frankl, F.},
      author={Pontrjagin, L.},
       title={Ein {K}notensatz mit {A}nwendung auf die {D}imensionstheorie},
        date={1930},
        ISSN={0025-5831},
     journal={Math. Ann.},
      volume={102},
      number={1},
       pages={785\ndash 789},
         url={https://doi.org/10.1007/BF01782377},
      review={\MR{1512608}},
}

\bib{HS}{article}{
      author={Hayden, K.},
      author={J.M., Sabloff},
       title={Positive knots and {L}agrangian fillability},
        date={2015},
     journal={Proc. Amer. Math. Soc.},
      volume={143},
       pages={1813\ndash 1821},
      review={\MR{3314092}},
}

\bib{knuth}{book}{
      author={Knuth, Donald~E.},
       title={The art of computer programming. {V}ol. 4{A}. {C}ombinatorial
  algorithms. {P}art 1},
   publisher={Addison-Wesley, Upper Saddle River, NJ},
        date={2011},
        ISBN={978-0-201-03804-0; 0-201-03804-8},
      review={\MR{3444818}},
}

\bib{Lickorish}{book}{
      author={Lickorish, W. B.~Raymond},
       title={An introduction to knot theory},
      series={Graduate Texts in Mathematics},
   publisher={Springer-Verlag, New York},
        date={1997},
      volume={175},
        ISBN={0-387-98254-X},
         url={https://doi.org/10.1007/978-1-4612-0691-0},
      review={\MR{1472978}},
}

\bib{NT}{inproceedings}{
      author={Ng, Lenhard},
      author={Thurston, Dylan},
       title={Grid diagrams, braids, and contact geometry},
        date={2009},
   booktitle={Proceedings of {G}\"okova {G}eometry-{T}opology {C}onference
  2008},
   publisher={G\"okova Geometry/Topology Conference (GGT), G\"okova},
       pages={120\ndash 136},
      review={\MR{2500576}},
}

\bib{OEISschroder}{misc}{
      author={of~Integer~Sequences, Online~Encyclopedia},
      editor={Sloane, N. J.~A.},
       title={Sequence {A006318}},
         how={https://oeis.org/A006318},
        note={Accessed: 2020-02-21},
}

\bib{peterson}{article}{
      author={Petersen, T.~Kyle},
      author={Tenner, Bridget~Eileen},
       title={The depth of a permutation},
        date={2015},
        ISSN={2156-3527},
     journal={J. Comb.},
      volume={6},
      number={1-2},
       pages={145\ndash 178},
         url={https://doi.org/10.4310/JOC.2015.v6.n1.a9},
      review={\MR{3338848}},
}

\bib{Seifert}{article}{
      author={Seifert, H.},
       title={\"{U}ber das {G}eschlecht von {K}noten},
        date={1935},
        ISSN={0025-5831},
     journal={Math. Ann.},
      volume={110},
      number={1},
       pages={571\ndash 592},
         url={https://doi.org/10.1007/BF01448044},
      review={\MR{1512955}},
}

\bib{SS}{article}{
      author={Shapiro, Louis},
      author={Stephens, A.~B.},
       title={Bootstrap percolation, the {S}chr\"oder numbers, and the
  {$N$}-kings problem},
        date={1991},
        ISSN={0895-4801},
     journal={SIAM J. Discrete Math.},
      volume={4},
      number={2},
       pages={275\ndash 280},
         url={https://doi.org/10.1137/0404025},
      review={\MR{1093199}},
}

\bib{Witte}{book}{
      author={Witte, Shawn~L.},
       title={Link nomenclature, random grid diagrams, and {M}arkov {C}hain
  methods in knot theory},
        date={2020},
        note={Ph.D. Dissertation {--} University of California-Davis},
}

\end{biblist}
\end{bibdiv}

\end{document}